\documentclass{article}

\usepackage[round]{natbib}

\usepackage{tikz}
\usepackage{pifont}

\usepackage[utf8]{inputenc} 
\usepackage[T1]{fontenc}    
\usepackage{url}            
\usepackage{booktabs}       
\usepackage{amsfonts}       
\usepackage{nicefrac}       
\usepackage{microtype}      
\usepackage{xcolor}         
\usepackage{float} 
\usepackage{graphicx} 
\usepackage{subfigure} 

\usepackage{bm} 

\usepackage[margin=1.03in]{geometry} 

\usepackage{algorithm} 
\usepackage{algorithmic}

\usepackage{wrapfig}

\usepackage{enumerate} 

\usepackage{amsmath}
\usepackage{amssymb}
\usepackage{mathtools}
\usepackage{amsthm}
\usepackage{mathrsfs}

\usepackage[hidelinks]{hyperref}

\newtheorem{theorem}{Theorem}

\newtheorem{lemma}{Lemma}
\newtheorem{remark}{Remark}

\newtheorem{assumption}{Assumption}

\newcommand{\F}{ \mathcal{F} }

\renewcommand{\Pr}{ \mathbb{P} }

\newcommand{\R}{\mathbb{R}}

\newcommand{\E}{\mathbb{E}}

\newcommand{\df}{\mathrm{d}}
\newcommand{\dX}{\dot{X}}
\newcommand{\ddX}{\ddot{X}}
\newcommand{\ene}{\mathcal{E}}
\newcommand{\lang}{\left\langle}
\newcommand{\rang}{\right\rangle}
\newcommand{\lnorm}{\left\|}
\newcommand{\rnorm}{\right\|}

\newcommand{\wX}{\widetilde{X}}
\newcommand{\wY}{\widetilde{Y}}
\newcommand{\wV}{\widetilde{V}}
\newcommand{\lb}{\left(}
\newcommand{\rb}{\right)}
\newcommand{\lbm}{\left[}
\newcommand{\rbm}{\right]}
\newcommand{\tr}{\mathsf{T}}

\newcommand{\hochkomma}{${\,\,}^{,}$}

\title{The \emph{Anytime} Convergence of Stochastic Gradient Descent with Momentum: From a Continuous-Time Perspective}

\author{
Yasong Feng\footnote{Equal first authors.}\hochkomma\thanks{Shanghai Center for Mathematical Sciences, Fudan University; email: \texttt{ysfeng20@fudan.edu.cn}.}\and
Yifan Jiang${}^{*,}$\thanks{Mathematical Institute, University of Oxford; email: \texttt{yifan.jiang@maths.ox.ac.uk}.}\and
Tianyu Wang\footnote{Correponding author.}\hochkomma\thanks{Shanghai Center for Mathematical Sciences, Fudan University; email: \texttt{wangtianyu@fudan.edu.cn}.}\and
Zhiliang Ying\thanks{Department of Statistics, Columbia University; email: \texttt{zying@stat.columbia.edu}.}}
\date{}
\begin{document}

\maketitle

\begin{abstract}
    We study the stochastic optimization problem from a continuous-time perspective, with a focus on the Stochastic Gradient Descent with Momentum (SGDM) method. We show that the trajectory of SGDM, despite its \emph{stochastic} nature, converges in $L_2$-norm to a \emph{deterministic} second-order Ordinary Differential Equation (ODE) as the stepsize goes to zero. The connection between the ODE and the algorithm results in a useful development for the discrete-time convergence analysis. More specifically, we develop, through the construction of a suitable Lyapunov function, convergence results for the ODE, which are then translated to the corresponding convergence results for the discrete-time case. This approach yields a novel \emph{anytime} convergence guarantee for stochastic gradient methods. In particular, we prove that the sequence $\{ x_k \}$, governed by running SGDM on a smooth convex function $f$, satisfies
    \begin{align*}
        \Pr\left(f (x_k) - f^* \le C\left(1+\log\frac{1}{\beta}\right)\frac{\log k}{\sqrt{k}},\;\text{for all $k$}\right)\ge 1-\beta\text{\quad for any $\beta>0$,}
    \end{align*}
    where $f^*=\min_{x\in\R^n} f(x)$, and $C$ is a constant. Rather than at a single step, this result captures the convergence behavior across the entire trajectory of the algorithm.
\end{abstract} 





\section{Introduction}

Stochastic gradient methods are widely used in various fields such as machine learning \citep[e.g.,][]{johnson2013accelerating, loizou2017linearly}, statistics \citep[e.g.,][]{robbins1951stochastic}, and more. These algorithms, including Stochastic Gradient Descent (SGD) and its variants, have been successfully implemented in many applications. Many research works have focused on establishing convergence guarantees for SGD \citep{nemirovski2009robust, shamir2013stochastic,harvey2019tight} and its variants \citep{lan2012optimal, ghadimi2013stochastic, gitman2019understanding, li2019convergence, gorbunov2020stochastic, liu2020improved}, aiming to ensure their efficiency and reliability. 

Also, the relationship between the continuous-time dynamics and the discrete-time recursive algorithms has been studied in mathematics and related fields \citep{ljung1977analysis, helmke1996optimization}. Recently, there is a resurgence of interest in the understanding of continuous-time interpretation of discrete-time algorithms, inspired by the ODE modeling of Nesterov's accelerated gradient method \citep{su2014differential}. This new perspective has led to valuable insights in the analysis \citep{wibisono2016variational, shi2021understanding} and design \citep{krichene2015accelerated, zhang2018direct,  shi2019acceleration, zhang2021revisiting} of the deterministic optimization algorithms.

The relationship between continuous-time dynamics and discrete-time algorithms has also been explored in the context of stochastic optimization \citep{mandt2016variational, li2017stochastic, he2018differential, orvieto2019continuous, cheng2020stochastic, shi2021hyperparameters, difonzo2022stochastic,shi2023learning}. In particular, previous work by \citet{cheng2020stochastic} investigated numeric SDEs in machine learning applications, motivated by stochastic gradient algorithms. 
However, the continuous-time limits of some stochastic optimization algorithms, including the SGDM in this paper, exhibit different behavior. Specifically, the stochasticity may disappear as the stepsize goes to zero. 
Another line of research has centered around Langevin dynamics \citep[e.g.,][]{hsieh2018mirrored,xu2018global}. With roots in the ergodic analysis of simulated annealing and related algorithms \citep{chiang1987diffusion,gelfand1991recursive,belisle1992convergence,granville1994simulated, mattingly2010convergence}, recent studies in the machine learning community have primarily focused on specific tasks, such as statistical sampling for machine learning \citep{welling2011bayesian,zou2019sampling,chatterji2020langevin,mou2021high} and Riemannian learning \citep{cheng2022efficient}.


In this paper, we consider the optimization problem, $\min_{x\in\R^n}f(x)$, in an environment where $f$ is accessible only through a stochastic gradient oracle. That is, if we query the oracle at $x$, it returns a stochastic version $g(x,\xi)$ of the gradient $\nabla f(x)$. We use $x^*$ to denote an optimal solution of this problem, and $f^*=f(x^*)=\min_{x\in\R^n}f(x)$ is the minimal function value.
Our main focus is on studying the Stochastic Gradient Descent with Momentum (SGDM) method given by
\begin{align}\label{eq:SGDM-intro}
    \begin{split}
        &x_k-x_{k-1}=\eta v_{k-1},\\
        &v_k-v_{k-1}=-\frac{2}{k}v_k-\frac{2}{k}\cdot\frac{1}{\sqrt{k\eta}}g(x_k,\xi_k), 
    \end{split}
\end{align}
where $\eta$ is the stepsize. From a continuous-time perspective, we rigorously prove that as $\eta$ tends to zero, the sequence $\{x_k\}$ generated by SGDM converges in $L_2$-norm to the trajectory of the second-order ODE
\begin{align}\label{eq:SGDM-ODE-intro}
    \ddot{X}(t)+\frac{2}{t}\dot{X}(t)+\frac{2}{t^{3/2}}\nabla f(X(t))=0.
\end{align} 
The continuous-time limits of some other stochastic optimization methods are also ODEs of the form (\ref{eq:SGDM-ODE-intro}), including the well-known AC-SA algorithm \citep{lan2020first}; see Appendix \ref{app:acsa} for a detailed derivation.

This continuous-time version provides insights into the discrete-time analysis. In this connection, we establish the convergence guarantees of the ODE \eqref{eq:SGDM-ODE-intro}, which are then translated into the convergence results of SGDM. Specifically, this approach gives a new characterization of the dynamical property of SGDM. Leveraging this dynamical insight, we develop a novel supermartingale argument and contribute to the existing theoretical analysis on stochastic gradient methods, particularly in terms of the anytime high probability convergence guarantee under mild conditions.

\subsection{The Anytime High Probability Convergence and Our Contributions}
For theoretical analysis of stochastic optimization algorithms, an important topic is to develop convergence guarantees that hold with high probability \citep{nemirovski2009robust, lan2012optimal, nazin2019algorithms, gorbunov2020stochastic, liu2023high}. This is partly due to the fact that high probability analysis better reflects the performance of a single run of the algorithm. For a given probability $1-\beta$, the existing results typically take the form:
\begin{align}\label{result-type-fixk}
    \text{\emph{For each $k\in\mathbb{N}$}}, \; \Pr\left(f(x_k)-f^*\leq U(k,\beta)\right)\geq1-\beta,
\end{align}
where $U(k,\beta)$ represents an upper bound depending on $k$ and $\beta$. However, the guarantee (\ref{result-type-fixk}) is insufficient if a practitioner wants to capture the behavior of the whole trajectory of the algorithm, or requires the upper bound to remain valid with high probability at any data-dependent stopping time \citep{waudby2020confidence, howard2021time}. 
For instance, multi-fidelity hyperparameter optimization (HPO) algorithms \citep{jamieson2016non, li2018hyperband} use losses at various training iterates, rather than only at the last iterate, in the hyperpameter selection, and thus accelerate the HPO process. Therefore, the convergence bound of the whole training loss trajectory is important for the effectiveness of the HPO algorithm.
In such cases, the following condition for the sequence $\{x_k\}$ must be met:
\begin{align}\label{result-type-anytime}
    \Pr\left(f(x_k)-f^*\leq U(k,\beta),\;\text{for all } k\right)\geq1-\beta.
\end{align}
We shall refer to results of the form (\ref{result-type-anytime}) as \emph{anytime high probability guarantees}.
Furthermore, this concept is closely related to the \emph{confidence sequence}, which has been widely investigated in sequential statistics \citep{lai1976confidence, sinha1984invariant} and applied to online learning \citep{jamieson2014lil, kaufmann2021mixture, orabona2023tight}. However, to the best of our knowledge, anytime convergence guarantees have not been studied in the stochastic optimization literature.




In this paper, we leverage a new continuous-time perspective to provide the anytime high probability result of the form (\ref{result-type-anytime}) for stochastic gradient methods. Notably, our \emph{anytime} convergence achieves the same rate as the state-of-the-art results for \emph{a single} iterate \citep{lan2012optimal,liu2023high}. Our main contributions are as follows.


\begin{enumerate}
    \item We are, to the best of our knowledge, the first to rigorously show that the continuous-time $L_2$-norm limit of a momentumized stochastic gradient method is a deterministic second-order ODE. We are the first to rigorously handle the quadratic variation process in analyzing the continuous-time limit of momentumized stochastic gradient method. This quadratic variation term is crucial in It\^{o}'s calculus. We develop convergence results for this ODE through a Lyapunov function, and then use it as a guidance for the algorithm analysis. 

    \item We translate the whole continuous-time argument to the discrete-time case to construct the Lyapunov function of SGDM, $\ene(k)$, and establish its descent property. This approach reveals the intrinsic dynamic property of SGDM. Furthermore, the increment of the Lyapunov function, $\ene(k+1)-\ene(k)$, is itself bounded by its previous value. As a consequence, our discrete-time analysis does not rely on any projection, or the assumption of bounded gradients or stochastic gradients.

    \item We establish the first explicit \emph{anytime} high probability convergence guarantee for stochastic gradient methods. We prove that the sequence $\{ x_k \}$ governed by running SGDM on a smooth convex function $f$ satisfies
    \begin{align*}\label{eq:anytime-intro}
        \Pr \left(f (x_k) - f^* \le\frac{C\log k\left(1+\log\frac{1}{\beta}\right)}{\sqrt{k}},\;\text{for all $k$}\right)
        \ge1 - \beta
    \end{align*}
    for any $\beta$, where $C$ is a constant. As far as we know, no prior research on stochastic gradient methods has studied the probability that the convergence bound holds simultaneously over iterates. In this work, we propose a novel procedure for proving this anytime high probability convergence without losing any factors in $k$, compared to its single-iterate counterpart. 
    More specifically, based on the descent property of the Lyapunov function, we derive a Gronwall-type inequality. Then our key approach is to construct a supermartingale by altering the moment generating function of $\ene(k)$, which leads to the anytime convergence bound.


    \item For smooth convex objective functions, we show that with probability $1$ there exists a sub-sequence $\{x_{k_l}\}$ of $\{x_k\}$ generated by SGDM satisfies $f(x_{k_l})-f^*=o\left(\frac{1}{\sqrt{k_l}\log\log k_l}\right)$. 
    This rate is new to the best of our knowledge.
\end{enumerate}

\subsection{Comparison to existing results}
We illustrate the qualitative and quantitative improvements of our \emph{anytime} high probability convergence results.
The high probability guarantees for stochastic gradient methods were obtained in  \cite{lan2012optimal, harvey2019tight, liu2023high, liu2023revisiting}. 
It has been shown that for a fix number of iterations $k$, with probability greater than $1-\beta$,
\begin{equation*}
 \frac{1}{k}\sum_{l=1}^{k}f(x_l)-f^*\leq U(k,\beta) \quad \text{and} \quad  f(x_k)-f^*\leq U(k,\beta), 
\end{equation*}
where $U(k,\beta)$\footnote{The bounds here are adapted to the setting of the current paper, where the learning rate is independent of $\beta$, see \cite[Theorem C.2]{liu2023revisiting} for full details.} is the asymptotic convergence rate given by  
\begin{equation*}
    U(k,\beta)=C\left(1+\log\frac{1}{\beta}\right)\frac{\log k}{\sqrt{k}}.
\end{equation*}
One may derive a naive anytime high probability guarantee by setting $\beta=\beta/k^{2}$ and taking the union bounds.
This indeed leads to 
\begin{equation*}
    \Pr\left(f(x_k)-f^*\leq \widetilde{U}(k,\beta),\;\text{for all $k$}\right)\geq 1-\mathcal{O}(\beta),
\end{equation*}
with 
\begin{equation*}
    \widetilde{U}(k,\beta)=U(k,\beta)\log k.
\end{equation*}
Our anytime high probability convergence guarantee, however, shows that the extra $\log k$  factor can be reduced.
The \emph{anytime} convergence rate we obtained in Theorem \ref{thm:real-anytime} is (surprisingly) of the same order as the ones for the average iterate and the last iterate.
Therefore, our \emph{anytime} convergence guarantee achieves an $\mathcal{O} (\log k)$ improvement compared to the state-of-the-art result \citep{liu2023high}. 
We remark that it is due to the novel energy estimate in Lemma \ref{lem:explicit-decay-difference-alpha} and an application of the supermartingale inequality.

Another potential approach to derive \emph{anytime} high probability convergence is utilizing the almost sure convergence results. \citet{sebbouh2021almost} derived almost sure convergence for various stochastic optimization algorithms, and in particular, the last-iterate almost sure convergence  of stochastic heavy-ball algorithm in Theorem 13.
By taking $\eta_k=\log k/\sqrt{k}$, it has shown that, \emph{almost surely}, $f(x_k)-f^*= o(\log k/\sqrt{k})$.
In other words, for any $\beta$, there exists $C(\beta)$ such that 
\begin{equation*}
    \Pr\left(f(x_k)-f^*\leq C(\beta)\frac{\log k}{\sqrt{k}},\;\text{for all $k$}\right)\geq 1-\beta.
\end{equation*}
However, it is unclear how the constant $C(\beta)$ depends on $\beta$.
Our results answer this question with an explicit $C(\beta)=C(1+\log(1/\beta))$ under the sub-Gaussian assumption on the noise. 

\subsection{Existing High Probability Results}\label{sec:literature}




There have been substantial advancements in establishing high probability guarantees for stochastic gradient methods. \citet{nemirovski2009robust} established convergence guarantees for a robust stochastic approximation approach. \citet{lan2012optimal} proposed an accelerated method that is universally optimal for solving non-smooth, smooth, and stochastic problems. \citet{ghadimi2013stochastic} developed
a two-phase method to improve the large deviation properties.
\citet{nazin2019algorithms}, \citet{madden2020high}, \citet{gorbunov2020stochastic}, \citet{cutkosky2021high}, and \citet{li2022high} proposed high probability convergence results without the light-tailed assumption.
\citet{li2020high} and \citet{kavis2022high} studied stochastic gradient methods with adaptive stepsize.
\citet{liu2023high} provided a new concentration argument to improve the high probability analysis.
The last-iterate high probability convergence results for SGD is studied by \citet{harvey2019tight}. \citet{jain2021making} introduced a stepsize sequence to obtain the optimal last-iterate convergence bound. \citet{liu2023revisiting} also presented a unified way for enhancing the last-iterate convergence analysis of stochastic gradient methods.
Yet all of the existing works provide high probability results that hold for individual $k$, and none of them study the anytime convergence, or more specifically, the probability that the convergence guarantee holds simultaneously over all values of $k$. A survey of more related works on stochastic optimization is provided in Appendix \ref{sec:additional}, including existing in expectation convergence results.

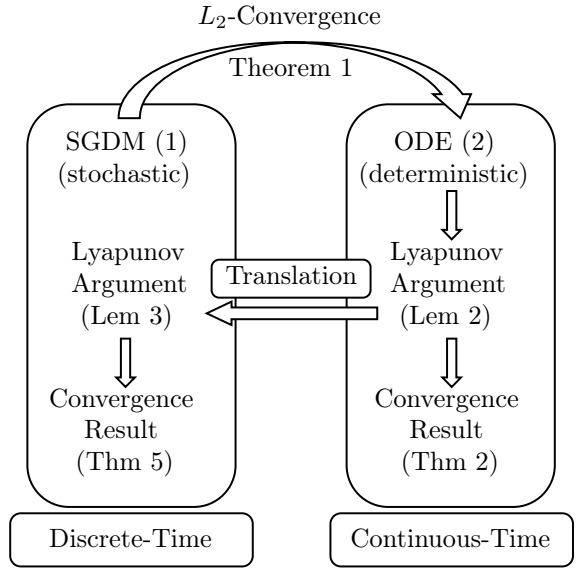
\begin{figure}[h]
    \centering
    \tikzset{every picture/.style={line width=0.75pt}} 

\begin{tikzpicture}[x=0.75pt,y=0.75pt,yscale=-1,xscale=1]

\draw   (13.83,77) .. controls (13.83,65.4) and (23.24,56) .. (34.83,56) -- (97.83,56) .. controls (109.43,56) and (118.83,65.4) .. (118.83,77) -- (118.83,238.22) .. controls (118.83,249.82) and (109.43,259.22) .. (97.83,259.22) -- (34.83,259.22) .. controls (23.24,259.22) and (13.83,249.82) .. (13.83,238.22) -- cycle ;
\draw   (5.17,267.78) .. controls (5.17,264.82) and (7.57,262.42) .. (10.53,262.42) -- (122.13,262.42) .. controls (125.1,262.42) and (127.5,264.82) .. (127.5,267.78) -- (127.5,283.88) .. controls (127.5,286.85) and (125.1,289.25) .. (122.13,289.25) -- (10.53,289.25) .. controls (7.57,289.25) and (5.17,286.85) .. (5.17,283.88) -- cycle ;
\draw   (166.67,267.78) .. controls (166.67,264.82) and (169.07,262.42) .. (172.03,262.42) -- (283.63,262.42) .. controls (286.6,262.42) and (289,264.82) .. (289,267.78) -- (289,283.88) .. controls (289,286.85) and (286.6,289.25) .. (283.63,289.25) -- (172.03,289.25) .. controls (169.07,289.25) and (166.67,286.85) .. (166.67,283.88) -- cycle ;
\draw   (174.83,77) .. controls (174.83,65.4) and (184.24,56) .. (195.83,56) -- (258.83,56) .. controls (270.43,56) and (279.83,65.4) .. (279.83,77) -- (279.83,238.22) .. controls (279.83,249.82) and (270.43,259.22) .. (258.83,259.22) -- (195.83,259.22) .. controls (184.24,259.22) and (174.83,249.82) .. (174.83,238.22) -- cycle ;
\draw  [fill={rgb, 255:red, 255; green, 255; blue, 255 }  ,fill opacity=1 ] (104.28,163.5 - 2) -- (116.79,157.58 - 2) -- (116.79,160.54 - 2) -- (190.28,160.54 - 2) -- (190.28,166.46 - 2) -- (116.79,166.46 - 2) -- (116.79,169.42 - 2) -- cycle ;
\draw  [color={rgb, 255:red, 0; green, 0; blue, 0 }  ,draw opacity=1 ][fill={rgb, 255:red, 255; green, 255; blue, 255 }  ,fill opacity=1 ] (106.75,140.25 - 2) .. controls (106.75,138.13 - 2) and (108.47,136.42 - 2) .. (110.58,136.42 - 2) -- (185 - 1.92,136.42 - 2) .. controls (187.12 - 1.92,136.42 - 2) and (188.83 - 1.92,138.13 - 2) .. (188.83 - 1.92,140.25 - 2) -- (188.83 - 1.92,151.75 - 2) .. controls (188.83 - 1.92,153.87 - 2) and (187.12 - 1.92,155.58 - 2) .. (185 - 1.92,155.58 - 2) -- (110.58,155.58 - 2) .. controls (108.47,155.58 - 2) and (106.75,153.87 - 2) .. (106.75,151.75 - 2) -- cycle ;
\draw  [fill={rgb, 255:red, 255; green, 255; blue, 255 }  ,fill opacity=1 ] (154.09,24.72) .. controls (107.95,24.72) and (70.55,41.86) .. (70.55,63) -- (59.83,63) .. controls (59.83,41.86) and (97.23,24.72) .. (143.37,24.72) ;\draw  [fill={rgb, 255:red, 255; green, 255; blue, 255 }  ,fill opacity=1 ] (143.37,24.72) .. controls (180.26,24.72) and (211.56,35.68) .. (222.63,50.88) -- (218.42,50.88) -- (232.26,63) -- (237.56,50.88) -- (233.35,50.88) .. controls (222.28,35.68) and (190.98,24.72) .. (154.09,24.72)(143.37,24.72) -- (154.09,24.72) ;
\draw   (232.53,119.36) -- (230.18,119.36) -- (230.18,99.76) -- (225.49,99.76) -- (225.49,119.36) -- (223.14,119.36) -- (227.83,124.17) -- cycle ;
\draw   (71.03 - 3,195.86 - 2) -- (68.68 - 3,195.86 - 2) -- (68.68 - 3,174.26) -- (63.99 - 3,174.26) -- (63.99 - 3,195.86 - 2) -- (61.64 - 3,195.86 - 2) -- (66.33 - 3,200.67 - 2) -- cycle ;
\draw   (232.53,195.86 - 2) -- (230.18,195.86 - 2) -- (230.18,174.26) -- (225.49,174.26) -- (225.49,195.86 - 2) -- (223.14,195.86 - 2) -- (227.83,200.67 - 2) -- cycle ;

\draw (27.33,67.5) node [anchor=north west][inner sep=0.75pt]   [align=left] {\ SGDM (\ref{eq:SGDM-intro})\\(stochastic)};
\draw (179.33,67.5) node [anchor=north west][inner sep=0.75pt]   [align=left] { \ \ \ \ ODE (\ref{eq:SGDM-ODE-intro})\\(deterministic)};
\draw (195.33,124) node [anchor=north west][inner sep=0.75pt]   [align=left] {Lyapunov \\Argument \\ \ (Lem \ref{lemma:energy-alpha})};
\draw (34.5,124) node [anchor=north west][inner sep=0.75pt]   [align=left] {Lyapunov\\Argument\\ \ (Lem \ref{lem:explicit-decay-difference-alpha})};
\draw (187.33,197.96) node [anchor=north west][inner sep=0.75pt]   [align=left] {Convergence\\ \ \ \ \ Result \\ \ \ \ (Thm \ref{thm:ode-rate})};
\draw (18.83,197.96) node [anchor=north west][inner sep=0.75pt]   [align=left] {\ Convergence\\ \ \ \ \ \ Result \\ \ \ \ \ (Thm \ref{thm:real-anytime})};
\draw (98.29,4.67) node [anchor=north west][inner sep=0.75pt]   [align=left] {$L_2$-Convergence};
\draw (113.79,31) node [anchor=north west][inner sep=0.75pt]   [align=left] {Theorem \ref{thm:sgdmtoode}};
\draw (112.29,138 - 2) node [anchor=north west][inner sep=0.75pt]   [align=left] {Translation};
\draw (23.33,268.83) node [anchor=north west][inner sep=0.75pt]   [align=left] {Discrete-Time};
\draw (175.33,268.83) node [anchor=north west][inner sep=0.75pt]   [align=left] {Continuous-Time};

\end{tikzpicture}
    \caption{The framework of our theoretical analysis.}
    \label{fig:structure}
\end{figure}

This paper is organized as follows. In Section \ref{sec:continuous}, we prove that the continuous-time limit of SGDM is the ODE (\ref{eq:SGDM-ODE-intro}), and then construct a Lyapunov function to give the convergence result of  (\ref{eq:SGDM-ODE-intro}). In Section \ref{sec: convergence-in-expectation-SGDM}, we translate the Lyapunov argument to the discrete case, which yields the discrete-time dynamical property. We then prove the convergence of SGDM in expectation. In addition, we demonstrate the almost sure convergence by the supermartingale convergence theorem. Moving on to Section \ref{sec:hp}, we utilize the above results and
establish the anytime high probability convergence guarantee of SGDM. We present the experimental results in Section \ref{sec:exp}, which confirm our theoretical findings. Figure \ref{fig:structure} illustrates the entire framework of our theoretical analysis.

\section{Continuous-Time Analysis} \label{sec:continuous}

In this section, we show that the continuous-time limit of SGDM is the ODE (\ref{eq:SGDM-ODE-intro}), and construct a Lyapunov function to derive the convergence result of (\ref{eq:SGDM-ODE-intro}). Our analysis relies on the following assumption on the objective function $f$.

\begin{assumption}\label{assump:convex-smooth}
    $f$ is convex and $L$-smooth, that is, $\|\nabla f(x)-\nabla f(y)\|\leq L\|x-y\|$ for any $x,y\in\R^n$. There exists $x^*=\arg\min_{x\in\R^n}f(x)$.
\end{assumption}

\subsection{Derivation of the ODE Model}\label{sec:derivation-ode}

We rewrite (\ref{eq:SGDM-ODE-intro}) as the following system
\begin{align}\label{eq:odesys}
    \left\{\begin{aligned}
        &\frac{\df X}{\df t}=V,\\
        &\frac{\df V}{\df t}=-\frac{2}{t}V-\frac{2}{{t}^\frac{3}{2}}\nabla f(X).
    \end{aligned}\right.
\end{align}
Theorem \ref{thm:sgdmtoode} shows that the stochastic sequence $\{x_k\}$ generated by SGDM converges to the trajectory of the deterministic ODE system (\ref{eq:odesys}) as the stepsize goes to zero. For our purpose, we consider the ODE system (\ref{eq:odesys}) starting at an arbitrary
$T_0 > 0$. 

\begin{theorem}\label{thm:sgdmtoode}
    Assume that $f$ is $L$-smooth. Let $(X(t), V(t))$ be the solution to the ODE system (\ref{eq:odesys}), and $\{(x_k, v_k)\}$ be generated from SGDM (\ref{eq:SGDM-intro}) with $g(x_k,\xi_k)=\nabla f(x_k)+\xi_k$, where $\xi_k\overset{iid}{\sim}\mathcal{N}(0,I_n)$. If $X(T_0)=x_{T_0/\eta}$ and $V(T_0)=v_{T_0/\eta}$ for $T_0>0$, then
    \begin{align*}
        \lim_{\eta\to0}\E\left[\|x_{T/\eta}-X(T)\|^2\right]=0
    \end{align*}
    holds for any $T>T_0$. Note that throughout we shell use the convention that $T_0/\eta$ and $T/\eta$ in subscripts of $x$ and $v$ are understood as their integer part.
\end{theorem}

The proof of Theorem \ref{thm:sgdmtoode} relies on constructing auxiliary SDEs with diminishing stochastic term, and stochastic calculus techniques calibrated to our setting. Before presenting the detailed proof, we shall give an informal derivation of the ODE (\ref{eq:SGDM-ODE-intro}) and (\ref{eq:odesys}).

We rewrite SGDM scheme (\ref{eq:SGDM-intro}) as
\begin{align}\label{eq:SGDMtoODE}
    \left\{\begin{aligned}
        &x_k-x_{k-1}=\eta v_{k-1},\\
        &v_k-v_{k-1}=-\frac{2}{k\eta}v_k\eta-\frac{2}{(k\eta)^\frac{3}{2}}\nabla f(x_k)\eta-\frac{2}{(k\eta)^\frac{3}{2}}(\sqrt{\eta}\xi_k)\sqrt{\eta}.
    \end{aligned}\right.
\end{align}
Let $W$ be a standard Brownian motion. Then we let $\eta$ in (\ref{eq:SGDMtoODE}) go to zero and replace $x_k-x_{k-1}$, $v_k-v_{k-1}$, $k\eta$, $\eta$ and $\sqrt{\eta}\xi$ by $\df X$, $\df V$, $t$, $\df t$ and $dW$, respectively. The discrete scheme (\ref{eq:SGDMtoODE}) becomes
\begin{align*}
    \left\{\begin{aligned}
        &\df X=V\df t,\\
        &\df V=-\frac{2}{t}V\df t-\frac{2}{{t}^\frac{3}{2}}\nabla f(X)\df t-\frac{2}{t^{\frac{3}{2}}}\sqrt{\df t}\df W.
    \end{aligned}\right.
\end{align*}
Since $\sqrt{\df t}\df W$ is a higher order term, we omit it and obtain (\ref{eq:odesys}), which leads to the ODE (\ref{eq:SGDM-ODE-intro}) by subscribing the first line to the second line.

\begin{proof}[Proof of Theorem \ref{thm:sgdmtoode}]
    For any fixed $\eta>0$, we write $t_k=k\eta$ and consider an auxiliary SDE system $\bigl(\wX(t),\wV(t)\bigr)$ on $[T_0, T]$ with initialization $\wX(T_0)=x_{T_0/\eta}$ and $\wV(T_0)=v_{T_0/\eta}$, which is defined as 
    \begin{align}\label{eq:auxsde}
    \left\{\begin{aligned}
        &\df\wX(t)=\wV(t_k)\df t,\\
        &\df\wV(t)=-\frac{2}{t_k}\wV(t_k)\df t-\frac{2}{t_k^\frac{3}{2}}\nabla f\left(\wX(t_k)\right)\df t-\frac{2}{t_k^{\frac{3}{2}}}\sqrt{\eta}dW(t),
    \end{aligned}\right.
    \end{align}
    for $t\in[t_k,t_{k+1})$. Integrating sequentially on each interval $[t_k,t_{k+1})$ shows that $\bigl(\wX(t_k),\wV(t_k)\bigr)$ and $(x_{k},v_k)$ are identically distributed for each $T_0\leq t_k\leq T$. For convenience, we denote 
    \begin{align*}
        \wY(t)=\begin{pmatrix}
            \wX(t)\\
            \wV(t)
        \end{pmatrix}\quad\text{and}\quad Y(t)=\begin{pmatrix}
            X(t)\\
            V(t)
        \end{pmatrix},
    \end{align*}
    where $X(t)$ and $V(t)$ are defined in the ODE system (\ref{eq:odesys}), and thus $\E\Bigl[\bigl\|x_{T/\eta}-X(T)\bigr\|^2\Bigr]=\E\Bigl[\bigl\|\wX(T)-X(T)\bigr\|^2\Bigl]\leq\E\Bigl[\bigl\|\wY(T)-Y(T)\bigr\|^2\Bigr]$.
    The following lemma shows that $\{\wY(t_k)\}$ are uniformly bounded in $L_2$-norm.

    \begin{lemma}\label{lem:boundwY}
        Let $[T_0, T]$ be a fixed horizon. There exists a constant $C_1$ such that for any $\eta>0$,
        \begin{equation}\label{eq:lem-bounded-1}
            \max_{\frac{T_0}{\eta}\leq k\leq\frac{T}{\eta}}\E\Bigl[\bigl\|\wY(t_k)\bigr\|^2\Bigr]\leq C_1,
        \end{equation}
        where $t_k=k\eta$. Furthermore, there exists a constant $C_2$ such that for any $\eta>0$ and any $T_0\leq t_k\leq t\leq t_{k+1}\leq T$,
        \begin{align}\label{eq:lem-bounded-2}
            \E\Bigl[\bigl\|\wY(t)-\wY(t_k)\bigr\|^2\Bigr]\leq C_2\eta^2.
        \end{align}
    \end{lemma}
    \begin{proof}
        From (\ref{eq:auxsde}), it follows that
        \begin{equation*}
            \wY(t)-\wY(t_k)=\begin{pmatrix}
                \wV(t_k)\\
                -\frac{2}{t_k}\wV(t_k)-\frac{2}{t_k^{\frac{3}{2}}}\nabla f(\wX(t_k))
            \end{pmatrix}\cdot(t-t_k)
            +
            \begin{pmatrix}
                0\\
                -\frac{2}{t_k^{\frac{3}{2}}}\sqrt{\eta}(W(t)-W(t_k))
            \end{pmatrix},
        \end{equation*}
        for any $t_k\leq t\leq t_{k+1}$. By the $L$-smoothness,
        \begin{equation}\label{eq:bound-nabla}
            \|\nabla f(x)\|=\|\nabla f(x)-\nabla f(x^*)\|\leq L\cdot\|x-x^*\|\leq L\|x\|+L\|x^*\|.
        \end{equation}
        Therefore, we have
        \begin{align*}
            \bigl\|\wY(t)-\wY(t_k)\bigr\|^2\leq&2\lb\bigl\|\wV(t_k)\bigr\|^2+\frac{8}{t_k^2}\bigl\|\wV(t_k)\bigr\|^2+\frac{16L^2}{t_k^3}\lb\bigl\|\wX(t_k)\bigr\|^2+\|x^*\|^2\rb\rb(t-t_k)^2\\
            &+\frac{8}{t_k^3}\eta(W(t)-W(t_k))^2,
        \end{align*}
        and thus there exists a constant $C_3$, depending on $T_0$, $\|x^*\|$ and $L$, such that
        \begin{equation*}
            \bigl\|\wY(t)-\wY(t_k)\bigr\|^2\leq C_3\lb\bigl\|\wY(t_k)\bigr\|^2+1\rb(t-t_k)^2+C_3\eta(W(t)-W(t_k))^2.
        \end{equation*}
        Since $t\leq t_{k+1}=t_k+\eta$, taking expectation on both sides yields
        \begin{equation}\label{eq:deriveymy}
            \E\Bigl[\bigl\|\wY(t)-\wY(t_k)\bigr\|^2\Bigr]\leq2C_3\eta^2\lb\E\Bigl[\bigl\|\wY(t_k)\bigr\|^2\Bigr]+1\rb.
        \end{equation}
        Specifically, by letting $t=t_{k+1}$ we have
        \begin{align*}
            \E\Bigl[\bigl\|\wY(t_{k+1})-\wY(t_k)\bigr\|^2\Bigr]\leq2C_3\eta^2\lb\E\Bigl[\bigl\|\wY(t_k)\bigr\|^2\Bigr]+1\rb.
        \end{align*}
        We denote 
        \begin{align*}
            a_n=\max_{\frac{T_0}{\eta}\leq k\leq\frac{T_0}{\eta}+n}\lb\E\Bigl[\bigl\|\wY(t_k)\bigr\|^2\Bigr]\rb^{\frac{1}{2}},
        \end{align*}
        and we have
        \begin{align*}
            a_{n+1}\leq a_n+\lb\E\Bigl[\bigl\|\wY(t_{k+1})-\wY(t_k)\bigr\|^2\Bigr]\rb^\frac{1}{2}&\leq a_n+\sqrt{2C_3}\eta\lb\E\Bigl[\bigl\|\wY(t_k)\bigr\|^2\Bigr]+1\rb^\frac{1}{2}\\
            &\leq(1+\sqrt{2C_3}\eta)a_n+\sqrt{2C_3}\eta.
        \end{align*}
        Therefore, for $N=\frac{T-T_0}{\eta}$, we have
        \begin{align*}
            a_N&\leq(1+\sqrt{2C_3}\eta)^Na_0+\sqrt{2C_3}\eta\sum_{k=0}^{N-1}(1+\sqrt{2C_3}\eta)^k\\
            &\leq\lb1+\frac{\sqrt{2C_3}(T-T_0)}{N}\rb^Na_0+\frac{\sqrt{2C_3}(T-T_0)}{N}\cdot N\lb1+\frac{\sqrt{2C_3}(T-T_0)}{N}\rb^N\\
            &\leq e^{\sqrt{2C_3}(T-T_0)}\lb\E\Bigl[\bigl\|\wY(T_0)\bigr\|^2\Bigr]\rb^{\frac{1}{2}}+\sqrt{2C_3}(T-T_0)e^{\sqrt{2C_3}(T-T_0)}.
        \end{align*}
        Thus, we finish the proof of (\ref{eq:lem-bounded-1}) by letting 
        \begin{equation*}
            C_1=\lb e^{\sqrt{2C_3}(T-T_0)}\lb\E\Bigl[\bigl\|\wY(T_0)\bigr\|^2\Bigr]\rb^{\frac{1}{2}}+\sqrt{2C_3}(T-T_0)e^{\sqrt{2C_3}(T-T_0)}\rb^2.
        \end{equation*}
        For the second part, we apply (\ref{eq:lem-bounded-1}) to the RHS of (\ref{eq:deriveymy}) to get
        \begin{align*}
            \E\Bigl[\bigl\|\wY(t)-\wY(t_k)\bigr\|^2\Bigr]\leq2C_3\eta^2\lb C_1+1\rb= C_2\eta^2,
        \end{align*}
        where $C_2=2C_3(C_1+1)$.
    \end{proof}
    In the remaining proof, we will upper bound
    \begin{equation}\label{eq:decompose}
        \df\E\Bigl[\bigl\|\wY(t)-Y(t)\bigr\|^2\Bigr]=\df\E\Bigl[\bigl\|\wX(t)-X(t)\bigr\|^2\Bigr]+\df\E\Bigl[\bigl\|\wV(t)-V(t)\bigr\|^2\Bigr],
    \end{equation}
    and then apply Gr\"{o}nwall's inequality to $\E\Bigl[\bigl\|\wY(t)-Y(t)\bigr\|^2\Bigr]$ to get the convergence result of Theorem \ref{thm:sgdmtoode}.

    For the first term on the RHS of (\ref{eq:decompose}), the It\^{o}'s formula yields that
    \begin{align*}
        \df\bigl\|\wX(t)-X(t)\bigr\|^2&=2\bigl(\wX(t)-X(t)\bigr)^\tr \df\bigl(\wX(t)-X(t)\bigr)+ \df\mathop{\mathrm{tr}}\bigl[\wX-X\bigr](t),
    \end{align*}
    where $\bigl[\wX-X\bigr]$ is the quadratic variation process of $\wX -X$.
    Since, by definitions (\ref{eq:odesys}) and (\ref{eq:auxsde}), $\df\bigl(\wX(t)-X(t)\bigr)=\bigl(\wV(t_k)-V(t)\bigr)\df t$, $\wX-X$ is absolutely continuous with respect to $\df t$. Therefore, its quadratic variation must be zero, i.e., $\bigl[\wX-X\bigr](t)=0$. We thus derive
    \begin{align}
        \begin{split}\label{eq:dwx-m-x2}
            &\df\bigl\|\wX(t)-X(t)\bigr\|^2\\
            =&2\bigl(\wX(t)-X(t)\bigr)^\tr\bigl(\wV(t_k)-\wV(t)\bigr)\df t + 2\bigl(\wX(t)-X(t)\bigr)^\tr\bigl(\wV(t)-V(t)\bigr)\df t\\
            \leq&\bigl\|\wX(t)-X(t)\bigr\|^2\df t+\bigl\|\wV(t_k)-\wV(t)\bigr\|^2\df t+\bigl\|\wX(t)-X(t)\bigr\|^2\df t+\bigl\|\wV(t)-V(t)\bigr\|^2\df t.
        \end{split}
    \end{align}
    From (\ref{eq:lem-bounded-2}) we have $\E\lbm\bigl\|\wV(t_k)-\wV(t)\bigr\|^2\rbm\leq C_2\eta^2$. Therefore, taking expectation on both sides of (\ref{eq:dwx-m-x2}) gives
    \begin{align}\label{eq:bound-wx-x2}
        \df\E\lbm\bigl\|\wX(t)-X(t)\bigr\|^2\rbm\leq3\E\lbm\bigl\|\wY(t)-Y(t)\bigr\|^2\rbm\df t+C_2\eta^2\df t.
    \end{align}

    For the second term on the RHS of (\ref{eq:decompose}), the It\^{o}'s formula yields that
    \begin{align}\label{eq:wv-v2-ito}
        \df\bigl\|\wV(t)-V(t)\bigr\|^2&=2\left(\wV(t)-V(t)\right)^\tr \df\left(\wV(t)-V(t)\right)+ \df\mathop{\mathrm{tr}}\bigl[\wV-V\bigr](t).
    \end{align}
    The definitions (\ref{eq:odesys}) and (\ref{eq:auxsde}) yields that
    \begin{align*}
        \df \lb\wV(t)-V(t)\rb=&-\lb\frac{2}{t_k}\wV(t_k)-\frac{2}{t}V(t)\rb\df t-\lb\frac{2}{t_k^{\frac{3}{2}}}\nabla f(\wX(t_k))-\frac{2}{t^{\frac{3}{2}}}\nabla f(X(t))\rb\df t\\
        &-\frac{2}{t_k^{\frac{3}{2}}}\sqrt{\eta}\df W(t).
    \end{align*}
    Then the definition of quadratic variation gives
    \begin{align}\label{bound-wv-v2-1}
        \df  \mathop{\mathrm{tr}}\bigl[\wV-V\bigr](t)=\frac{4n}{t_k^3}\eta\df t\leq\frac{4n}{T_0^3}\eta\df t.
    \end{align} 
    We also have
    \begin{align}\label{bound-wv-v2-2}
        \begin{split}
            2\left(\wV(t)-V(t)\right)^\tr \df\left(\wV(t)-V(t)\right)=&\underbrace{-\lb\wV(t)-V(t)\rb^\tr\lb\frac{4}{t_k}\wV(t_k)-\frac{4}{t}V(t)\rb\df t}_{\text{\ding{172}}}\\
            &\underbrace{-\lb\wV(t)-V(t)\rb^\tr\lb\frac{4}{t_k^{\frac{3}{2}}}\nabla f(\wX(t_k))-\frac{4}{t^{\frac{3}{2}}}\nabla f(X(t))\rb\df t}_{\text{\ding{173}}}\\
            &\underbrace{-\lb\wV(t)-V(t)\rb^\tr\cdot\frac{4}{t_k^{\frac{3}{2}}}\sqrt{\eta}\df W(t)}_{\text{\ding{174}}}.
        \end{split}
    \end{align}
    The first part, \ding{172}, equals
    \begin{align}
        \begin{split}\label{bound-ne-ding1}
            &-\lb\wV(t)-V(t)\rb^\tr\lb\frac{4}{t_k}\wV(t_k)-\frac{4}{t}\wV(t)+\frac{4}{t}\wV(t)-\frac{4}{t}V(t)\rb\df t\\
            =&-\lb\wV(t)-V(t)\rb^\tr\lb\frac{4}{t_k}\wV(t_k)-\frac{4}{t}\wV(t)\rb\df t-\lb\wV(t)-V(t)\rb^\tr\lb\frac{4}{t}\wV(t)-\frac{4}{t}V(t)\rb\df t\\
            \leq&\bigl\|\wV(t)-V(t)\bigr\|^2\df t+\bigl\|\frac{2}{t_k}\wV(t_k)-\frac{2}{t}\wV(t)\bigr\|^2\df t+\bigl\|\wV(t)-V(t)\bigr\|^2\df t+\frac{4}{T_0^2}\bigl\|\wV(t)-V(t)\bigr\|^2\df t.
        \end{split}
    \end{align}
    We note that
    \begin{align*}
        \lnorm\frac{2}{t_k}\wV(t_k)-\frac{2}{t}\wV(t)\rnorm^2\df t=&\lnorm\frac{2}{t_k}\wV(t_k)-\frac{2}{t}\wV(t_k)+\frac{2}{t}\wV(t_k)-\frac{2}{t}\wV(t)\rnorm^2\df t\\
        \leq&2\lb\frac{2\eta}{T_0^2}\rb^2\lnorm\wV(t_k)\rnorm^2\df t+\frac{4}{T_0^2}\lnorm\wV(t_k)-\wV(t)\rnorm^2\df t.
    \end{align*}
    Lemma \ref{lem:boundwY} gives that $\E\Bigl[\bigl\|\wV(t_k)\bigr\|^2\Bigr]\leq C_1$ and $\E\Bigl[\bigl\|\wV(t_k)-\wV(t)\bigr\|^2\Bigr]\leq C_2\eta^2$. So, by taking expectation on both sides of the above inequality, we further have
    \begin{align}\label{eq:bound-increase}
        \E\lbm\lnorm\frac{2}{t_k}\wV(t_k)-\frac{2}{t}\wV(t)\rnorm^2\rbm\df t\leq\lb\frac{8C_1}{T_0^4}+\frac{4C_2}{T_0^2}\rb\eta^2\df t.
    \end{align}
    Therefore, we denote $C_4=\frac{8C_1}{T_0^4}+\frac{4C_2}{T_0^2}$, and (\ref{bound-ne-ding1}) gives
    \begin{align}\label{eq:bound-e-ding1}
        \E[\text{\ding{172}}]\leq\lb2+\frac{4}{T_0^2}\rb\E\Bigl[\bigl\|\wY(t)-Y(t)\bigr\|^2\Bigr]\df t+C_4\eta^2\df t.
    \end{align}
    For \ding{173}, we have
    \begin{align}
        \begin{split}\label{bound-ne-ding2}
            \text{\ding{173}}=&-\lb\wV(t)-V(t)\rb^\tr\lb
            \frac{4}{t_k^{\frac{3}{2}}}\nabla f(\wX(t_k))
            -\frac{4}{t^{\frac{3}{2}}}\nabla f(\wX(t))
            \rb\df t\\
            &-\lb\wV(t)-V(t)\rb^\tr\lb
            \frac{4}{t^{\frac{3}{2}}}\nabla f(\wX(t))
            -\frac{4}{t^{\frac{3}{2}}}\nabla f(X(t))\rb\df t\\
            \leq&
            \lnorm\wV(t)-V(t)\rnorm^2\df t+
            \lnorm\frac{2}{t_k^{\frac{3}{2}}}\nabla f(\wX(t_k))
            -\frac{2}{t^{\frac{3}{2}}}\nabla f(\wX(t))\rnorm^2\df t\\
            &+
            \lnorm\wV(t)-V(t)\rnorm^2\df t+
            \frac{4L^2}{T_0^3}\lnorm\wX(t)
            -X(t)\rnorm^2\df t.
        \end{split}
    \end{align}
    Analogous to (\ref{eq:bound-increase}) and (\ref{eq:bound-nabla}), there exists constant $C_5$ such that
    \begin{align*}
        \E\lbm\lnorm\frac{2}{t_k^{\frac{3}{2}}}\nabla f(\wX(t_k))
            -\frac{2}{t^{\frac{3}{2}}}\nabla f(\wX(t))\rnorm^2\rbm\df t\leq C_5\eta^2\df t.
    \end{align*}
    Therefore, (\ref{bound-ne-ding2}) gives
    \begin{align}\label{eq:bound-e-ding2}
        \E[\text{\ding{173}}]\leq\lb2+\frac{4L}{T_0^3}\rb\E\Bigl[\bigl\|\wY(t)-Y(t)\bigr\|^2\Bigr]\df t+C_5\eta^2\df t.
    \end{align}
    Additionally, by Lemma \ref{lem:boundwY} we show \text{\ding{174}} is a true martingale, and thus we derive
    \begin{equation}\label{eq:bound-e-ding3}
        \E[\text{\ding{174}}]=0.
    \end{equation}
    As a consequence, taking expectation on both sides of (\ref{eq:wv-v2-ito}) and substituting (\ref{bound-wv-v2-1}), (\ref{bound-wv-v2-2}), (\ref{eq:bound-e-ding1}), (\ref{eq:bound-e-ding2}), and (\ref{eq:bound-e-ding3}) give
    \begin{align}\label{eq:bound-wv-v2}
        \df\E\Bigl[\bigl\|\wV(t)-V(t)\bigr\|^2\Bigr]\leq\lb4+\frac{4}{T_0^2}+\frac{4L}{T_0^3}\rb\E\Bigl[\bigl\|\wY(t)-Y(t)\bigr\|^2\Bigr]\df t+\lb C_4\eta^2+C_5\eta^2+\frac{4n}{T_0^3}\eta\rb\df t.
    \end{align}
    
    Substituting (\ref{eq:bound-wx-x2}) and (\ref{eq:bound-wv-v2}) to (\ref{eq:decompose}) yields that
    \begin{align*}
        \df\E\Bigl[\bigl\|\wY(t)-Y(t)\bigr\|^2\Bigr]\leq&\lb7+\frac{4}{T_0^2}+\frac{4L}{T_0^3}\rb\E\Bigl[\bigl\|\wY(t)-Y(t)\bigr\|^2\Bigr]\df t\\
        &+\lb C_2\eta^2+C_4\eta^2+C_5\eta^2+\frac{4n}{T_0^3}\eta\rb\df t.
    \end{align*}
    Without loss of generality, we let $\eta<1$. Thus, there exist constants $C_6$ and $C_7$ such that
    \begin{equation*}
        \df\E\Bigl[\bigl\|\wY(t)-Y(t)\bigr\|^2\Bigr]\leq C_6\E\Bigl[\bigl\|\wY(t)-Y(t)\bigr\|^2\Bigr]\df t+C_7\eta\df t.
    \end{equation*}
    The Gr\"{o}nwall's inequality gives
    \begin{equation*}
        \E\Bigl[\bigl\|\wY(T)-Y(T)\bigr\|^2\Bigr]-\E\Bigl[\bigl\|\wY(T_0)-Y(T_0)\bigr\|^2\Bigr]\leq\frac{C_7}{C_6}\lb e^{C_6(T-T_0)}-1\rb\eta.
    \end{equation*}
    Since $\wY(T_0)=Y(T_0)$ and $\E\Bigl[\|x_{T/\eta}-X(T)\|^2\Bigr]=\E\Bigl[\bigl\|\wX(T)-X(T)\bigr\|^2\Bigr]\leq\E\Bigl[\bigl\|\wY(T)-Y(T)\bigr\|^2\Bigr]$, we complete the proof.
\end{proof}



\subsection{Convergence Results of the ODE}

Here we consider a generalized version of ODE (\ref{eq:SGDM-ODE-intro})
\begin{equation}\label{ode:alpha}
    \ddot{X}(t)+\frac{p+1}{t}\dot{X}(t)+\frac{p+1}{t^\alpha}\nabla f(X(t))=0,
\end{equation}
where $\alpha$ and $p$ are parameters.
Parameter $\alpha$ models the speed of the gradient-decaying. When setting $\alpha=\frac{3}{2}$ and $p=1$, we recover ODE (\ref{eq:SGDM-ODE-intro}).

To describe the convergence of $X(t)$, we define the Lyapunov function
\begin{equation}\label{energy:alpha}
    \ene(t)=\lnorm pX+t\dot{X}-px^*\rnorm^2+2(p+1)t^{2-\alpha}(f(X)-f^*).
\end{equation}
The following lemma shows that $\ene(t)$ decreases monotonically with respect to $t$.
\begin{lemma}\label{lemma:energy-alpha}
    Let Assumption \ref{assump:convex-smooth} hold, and $X$ be the solution to (\ref{ode:alpha}) with $\alpha\leq2$ and $p+\alpha\geq2$. Then the Lyapunov function (\ref{energy:alpha}) satisfies
    \begin{equation*}
        \frac{\df\ene(t)}{\df t}\leq0.
    \end{equation*}
\end{lemma}
\begin{proof}
    By definition (\ref{energy:alpha}),
    \begin{align}\label{eq:continuous-diff}
        \begin{split}
            \frac{\df\ene(t)}{\df t}=&2\left\langle pX+t\dX-px^*,(p+1)\dX+t\ddX\right\rangle
            +2(p+1)t^{2-\alpha}\left\langle\dX,\nabla f(X)\right\rangle\\
            &+(4-2\alpha)(p+1)t^{1-\alpha}(f(X)-f^*).
        \end{split}
    \end{align}
    Since by (\ref{ode:alpha}), $(p+1)\dX+t\ddX=-(p+1)t^{1-\alpha}\nabla f(X)$, the above equality leads to
    \begin{align*}
        \frac{\df\ene(t)}{\df t}
        =&-2p(p+1)t^{1-\alpha}\left\langle X-x^*, \nabla f(X)\right\rangle
        +(4-2\alpha)(p+1)t^{1-\alpha}(f(X)-f^*)\\
        =&2(p+1)t^{1-\alpha}\bigg((2-\alpha)(f(X)-f^*)
        -p\left\langle X-x^*, \nabla f(X)\right\rangle\bigg).
    \end{align*} 
    From convexity of $f$, we have $f^*\geq f(X)+\lang\nabla f(X),x^*-X\rang$. This inequality indicates that $\lang\nabla f(X),x^*-X\rang \le 0 $. Then the condition $\alpha\leq2$ and $p+\alpha\geq2$ yields that
    \begin{align*}
        (2-\alpha)(f(X)-f^*)-p\left\langle X-x^*, \nabla f(X)\right\rangle
        \leq(2-\alpha)(f(X)-f^*-\left\langle X-x^*, \nabla f(X)\right\rangle)\leq0,
    \end{align*}
    and thus $\frac{\df\ene(t)}{\df t}\leq0$.
\end{proof}
With Lemma \ref{lemma:energy-alpha} in place, we characterize the convergence rate of the ODE (\ref{ode:alpha}) in Theorem \ref{thm:ode-rate} below.

\begin{theorem}
    \label{thm:ode-rate}
    Let Assumption \ref{assump:convex-smooth} hold, and $X$ be the solution of (\ref{ode:alpha}) with $\alpha\leq2$ and $p+\alpha\geq2$. Then for any $T>T_0>0$,
    \begin{equation*}
        f(X(T))-f^*\leq\frac{\ene(T_0)}{2(p+1)T^{2-\alpha}}.
    \end{equation*}
\end{theorem}
\begin{proof}
    Lemma \ref{lemma:energy-alpha} implies that $\ene(T)\leq\ene(T_0)$ for $T \ge T_0$. Therefore, we have
    \begin{equation*}
        f(X(T))-f^*\leq\frac{\ene(T)}{2(p+1)T^{2-\alpha}}\leq\frac{\ene(T_0)}{2(p+1)T^{2-\alpha}},
    \end{equation*}
    and the proof is completed.
\end{proof}

\section{Convergence Results of SGDM}\label{sec: convergence-in-expectation-SGDM}

In this section, we carry our continuous-time constructions over to the analysis of the discrete-time case. In this connection, we consider the following SGDM method
\begin{align}\label{eq:SGDM}
    x_{k+1}=x_k+\frac{k}{k+2}(x_k-x_{k-1})-\frac{2\sqrt{\eta_k}}{(k+2)\sqrt{k}}g(x_k,\xi_k).
\end{align}
The scheme (\ref{eq:SGDM}) is obtained from (\ref{eq:SGDM-intro}) by subscribing the first equality to the second equality and replacing $\eta$ by $\eta_k$. The choice of stepsize $\eta_k$ is presented below that does not rely on knowledge of the total number of iterates or the noise variance. 

We use $\mathcal{F}_k$ to denote the $\sigma$-algebra generated by all randomness after arriving at $x_k$, that is, $\mathcal{F}_k=\sigma(\xi_1,\cdots,\xi_{k-1})$. We introduce the following assumption for the discrete-time convergence analysis.

\begin{assumption}\label{assump:sto-gradient}
    For any $k\geq1$, the stochastic gradient $g(x_k,\xi_k)$ satisfies $\E[g(x_k, \xi_k)|\mathcal{F}_k]=\nabla f(x_k)$ and $\E[\|g(x_k,\xi_k)\|^2|\mathcal{F}_k]\leq\|\nabla f(x_k)\|^2+\sigma^2$.
\end{assumption}



		

\subsection{Discrete-Time Lyapunov Function}

Now we show how to translate the Lyapunov argument in Section \ref{sec:continuous} to the discrete-time case. Firstly, we define the discrete-time Lyapunov function
\begin{align}\label{discrete-energy:alpha}
        \ene(k)=\lnorm x_{k+1}+(k+1)(x_{k+1}-x_k)-x^*\rnorm^2
        +4\sqrt{(k+1)\eta_k}(f(x_{k})-f^*),
\end{align}
which is obtained from (\ref{energy:alpha}) by replacing $t$ with $(k+1)\eta_k$, and $\dot{X}$ with $v_{k+1}=\frac{x_{k+1}-x_k}{\eta_{k}}$. The following result shows the descent property of the Lyapunov function $\ene(k)$, and the proof of Lemma \ref{lem:explicit-decay-difference-alpha} is translated from that of Lemma \ref{lemma:energy-alpha} in the same way as the construction of $\ene(k)$.

\begin{lemma}\label{lem:explicit-decay-difference-alpha}
    Let Assumption \ref{assump:convex-smooth} hold. If the stepsize sequence $\{\eta_k\}$ is monotonically decreasing, then $\{x_k\}$ generated by SGDM (\ref{eq:SGDM}) satisfies
    \begin{align}
        \begin{split}\label{ene-decay-expectation-alpha}
            \ene(k)-\ene(k-1)
            \leq&\frac{4\eta_k}{k}\lnorm g(x_{k},\xi_{k})\rnorm^2-\frac{2}{L}\sqrt{\frac{\eta_k}{k}}\|\nabla f(x_{k})\|^2-2\sqrt{\frac{\eta_k}{k}}(f(x_{k})-f^*)\\
            &+4\sqrt{\frac{\eta_k}{k}}\langle\nabla f(x_{k})-g(x_{k},\xi_{k}),\tau_{k}\rangle
        \end{split}
    \end{align}
    for any $k\geq1$, where $\tau_k=k(x_{k}-x_{k-1})+(x_{k}-x^*)$.
\end{lemma}

We note that the first term $\frac{4\eta_k}{k}\lnorm g(x_{k},\xi_{k})\rnorm^2$ on the RHS of (\ref{ene-decay-expectation-alpha}) is a higher-order term with respect to $\eta_k$ than the second term $\frac{2}{L}\sqrt{\frac{\eta_k}{k}}\|\nabla f(x_{k})\|^2$; so this positive term can be controlled when $\eta_k$ is small. Next, the last term $4\sqrt{\frac{\eta_k}{k}}\langle\nabla f(x_{k})-g(x_{k},\xi_{k}),\tau_{k}\rangle$ has expectation equal to zero. 
Hence, $\ene(k)$ trends downward with respect to $k$, making Lemma \ref{lem:explicit-decay-difference-alpha} the discrete-time counterpart of Lemma \ref{lemma:energy-alpha}.

\begin{proof}[Proof of Lemma \ref{lem:explicit-decay-difference-alpha}]
    By differnencing the Lyapunov function (\ref{discrete-energy:alpha}), we get
    \begin{align}\label{eq:discrete-diff}
        \begin{split}
            &\ene(k)-\ene(k-1)\\
        \leq&\lnorm x_{k+1}+(k+1)(x_{k+1}-x_{k})-x^*\rnorm^2+4\sqrt{(k+1)\eta_{k}}(f(x_{k})-f^*)\\
        &-\lnorm x_{k}+k(x_{k}-x_{k-1})-x^*\rnorm^2-4\sqrt{k\eta_{k}}(f(x_{k-1})-f^*)\\
        \leq&2\langle 2(x_{k+1}-x_{k})+k(x_{k+1}-2x_{k}+x_{k-1}),
        x_{k+1}+(k+1)(x_{k+1}-x_{k})-x^*\rangle\\
        &-\lnorm 2(x_{k+1}-x_{k})+k(x_{k+1}-2x_{k}+x_{k-1})\rnorm^2+4\sqrt{k\eta_{k}}(f(x_{k})-f(x_{k-1}))\\
        &+2\sqrt{\frac{\eta_k}{k}}(f(x_{k})-f^*),
        \end{split}
    \end{align}
    where the first inequality follows from $\eta_{k}\leq\eta_{k-1}$, and the second inequality follows from $\|a\|^2-\|b\|^2=2\langle a-b,a\rangle-\|a-b\|^2$ and $\sqrt{k+1}-\sqrt{k}\leq\frac{1}{2\sqrt{k}}$. From (\ref{eq:SGDM}) we have
    \begin{align}\label{discrete-alpha}
        \begin{split}
            &2(x_{k+1}-x_{k})+k(x_{k+1}-2x_{k}+x_{k-1})\\
            =&(k+2)(x_{k+1}-x_{k})-k(x_{k}-x_{k-1})
        =-2\sqrt{\frac{\eta_k}{k}}g(x_{k}, \xi_{k}),
        \end{split}
    \end{align}
    and thus
    \begin{align}\label{explicit-decay-mid1}
        \begin{split}
            &\ene(k)-\ene(k-1)\\
            \leq&-4\left\langle\sqrt{\frac{\eta_k}{k}}g(x_{k}, \xi_{k}), x_{k}+(k+2)(x_{k+1}-x_{k})-x^*\right\rangle-\lnorm2\sqrt{\frac{\eta_k}{k}}g(x_{k}, \xi_{k})\rnorm^2\\
            &+4\sqrt{k\eta_{k}}(f(x_{k})-f(x_{k-1}))+2\sqrt{\frac{\eta_k}{k}}(f(x_{k})-f^*)\\
            =&-4\left\langle\sqrt{\frac{\eta_k}{k}} g(x_{k},\xi_{k}),(k+2)(x_{k+1}-x_{k})\right\rangle-\lnorm2\sqrt{\frac{\eta_k}{k}} g(x_{k},\xi_{k})\rnorm^2\\
            &+4\sqrt{k\eta_k}(f(x_{k})-f(x_{k-1}))+4\sqrt{\frac{\eta_k}{k}}(f(x_{k})-f^*-\langle g(x_{k},\xi_{k}), x_{k}-x^*\rangle)\\
            &-2\sqrt{\frac{\eta_k}{k}}(f(x_{k})-f^*).
        \end{split}
    \end{align}
    Intuitively, (\ref{eq:discrete-diff}) can be seen as the discrete-time counterpart of (\ref{eq:continuous-diff}), and (\ref{discrete-alpha}) corresponds to the ODE (\ref{ode:alpha}). Then we note that the convexity and $L$-smoothness of $f$ gives
    \begin{equation}\label{explicit-decay-convex1}
        f^*-f(x_{k})-\langle\nabla f(x_{k}), x^*-x_{k}\rangle\geq\frac{1}{2L}\|\nabla f(x_k)\|^2,
    \end{equation}
    and the convexity of $f$ shows that 
    \begin{equation}\label{explicit-decay-convex2}
        f(x_{k})-f(x_{k-1})\leq\langle\nabla f(x_{k}), x_{k}-x_{k-1}\rangle.
    \end{equation}
    Plugging (\ref{explicit-decay-convex1}) and (\ref{explicit-decay-convex2}) into (\ref{explicit-decay-mid1}), we have
    \begin{align*}
        &\ene(k)-\ene(k-1)\\
        \leq&4\sqrt{\frac{\eta_k}{k}}\langle g(x_k,\xi_k),-(k+2)x_{k+1}+(2k+2)x_k-kx_{k-1}\rangle\\
        &+4\sqrt{\frac{\eta_k}{k}}\langle\nabla f(x_{k})-g(x_{k},\xi_{k}),k(x_k-x_{k-1})\rangle\\
        &-\lnorm2\sqrt{\frac{\eta_k}{k}} g(x_{k},\xi_{k})\rnorm^2-\frac{2}{L}\sqrt{\frac{\eta_k}{k}}\|\nabla f(x_{k})\|^2
        +4\sqrt{\frac{\eta_k}{k}}\langle\nabla f(x_{k})-g(x_{k},\xi_{k}),x_{k}-x^*\rangle\\
        &-2\sqrt{\frac{\eta_k}{k}}(f(x_{k})-f^*).\nonumber
    \end{align*}
    Finally, we substitute (\ref{discrete-alpha}) to the RHS of the above inequality, and reorder the terms to conclude
    \begin{align*}
        \ene(k)-\ene(k-1) 
        \leq&\frac{4\eta_k}{k}\lnorm g(x_{k},\xi_{k})\rnorm^2-\frac{2}{L}\sqrt{\frac{\eta_k}{k}}\|\nabla f(x_{k})\|^2-2\sqrt{\frac{\eta_k}{k}}(f(x_{k})-f^*)\\
        &+4\sqrt{\frac{\eta_k}{k}}\langle\nabla f(x_{k})-g(x_{k},\xi_{k}), \tau_k\rangle,
    \end{align*}
    where $\tau_k=k(x_{k}-x_{k-1})+p(x_{k}-x^*)$.
\end{proof}

\subsection{Convergence in Expectation of SGDM}
With Lemma \ref{lem:explicit-decay-difference-alpha} in place, we present the convergence result in expectation of SGDM. Furthermore, in the following analysis we show that there exists a subsequence of $\{x_k\}$ converges at rate $o\left(\frac{1}{\sqrt{k}\log\log k}\right)$, which is \emph{faster than the lower bound} $\Omega\left(\frac{1}{\sqrt{k}}\right)$ of the whole sequence \citep{agarwal2012information}.
\begin{theorem}
    \label{thm:algo-rate}
    Let Assumptions \ref{assump:convex-smooth} and \ref{assump:sto-gradient} hold. Then $\{x_k\}$ generated by SGDM (\ref{eq:SGDM}) with stepsize $\eta_k=\frac{1}{L^2\log^2(k+2)}$ and initialization $x_0=x_1$ satisfies
    \begin{align*}
        \E[f(x_k)]-f^*\leq\frac{\left(3L\|x_0-x^*\|^2+\frac{4\sigma^2}{L}\right)\log(k+2)}{2\sqrt{k+1}},
    \end{align*}
    and
    \begin{align}\label{discrete-result-exp2}
        \liminf_{k\to\infty}\bigg(\sqrt{k}\log\log(k+2)\bigg)\bigg(\E[f(x_k)]-f^*\bigg)=0.
    \end{align}
\end{theorem}
\begin{proof}
    Since $\tau_k$ is $\mathcal{F}_{k}$-measurable, Assumption \ref{assump:sto-gradient} gives
    \begin{align*}
        \E[\langle\nabla f(x_{k})-g(x_{k},\xi_{k}),\tau_{k}\rangle|\mathcal{F}_{k}]
        =0
    \end{align*}
    and
    \begin{align*}
        \E[\|g(x_{k},\xi_{k})\|^2|\mathcal{F}_{k}]\leq\|\nabla f(x_{k})\|^2+\sigma^2. 
    \end{align*}
    By letting $\eta_k=\frac{1}{4L^2\log^2(k+2)}$, we have $\frac{4\eta_{k}}{k}\leq\frac{2}{L}\sqrt{\frac{\eta_{k}}{k}}$ for any $k\geq1$. Then we take conditional expectation with respect to $\mathcal{F}_{k}$ on both sides of (\ref{ene-decay-expectation-alpha}) to get
    \begin{align*}
        \E[\ene(k)-\ene(k-1)|\mathcal{F}_{k}]\leq\frac{4\eta_{k}}{k}\sigma^2-2\sqrt{\frac{\eta_{k}}{k}}(f(x_{k})-f^*).
    \end{align*}
    Consequently, taking expectation on both sides yields that
    \begin{align}\label{ene-decay-expresult-alpha}
        \begin{split}
            \E[\ene(k)]-\E[\ene(k-1)]
        \leq\frac{4\eta_{k}}{k}\sigma^2-2\sqrt{\frac{\eta_{k}}{k}}(\E[f(x_{k})]-f^*),
        \end{split}
    \end{align}
    for any $k\geq1$. Summing (\ref{ene-decay-expresult-alpha}) over $k$ yields that
    \begin{align}\label{ene-decay-expresult-alpha-1}
        \E[\ene(k)]\leq\ene(0)+4\sigma^2\sum_{i=1}^k\frac{\eta_{i}}{i}.
    \end{align}
    Additionally, the definition of the Lyapunov function (\ref{discrete-energy:alpha}) gives $4\sqrt{(k+1)\eta_k}(\E[f(x_k)]-f^*)\leq\E[\ene(k)]$ and
    \begin{align*}
        \ene(0)=\|x_0-x^*\|^2+\frac{2}{\log2\cdot L}(f(x_0)-f^*)
        \leq\left(1+\frac{2}{\log2\cdot L}\cdot\frac{L}{2}\right)\|x_0-x^*\|^2\leq3\|x_0-x^*\|^2.
    \end{align*}
    Substituting these two inequalities to (\ref{ene-decay-expresult-alpha-1}) implies that
    \begin{align*}
        \E[f(x_k)]-f^*
        \leq\frac{3\|x_0-x^*\|^2+\frac{\sigma^2}{L^2}\sum_{i=1}^k\frac{1}{i\log^2(i+2)}}{4\sqrt{\frac{k+1}{4L^2\log^2(k+2)}}}
        \leq\frac{\left(3L\|x_0-x^*\|^2+\frac{4\sigma^2}{L}\right)\log(k+2)}{2\sqrt{k+1}},
    \end{align*}
    where the second inequality follows from $\sum_{i=1}^\infty\frac{1}{i\log^2(i+2)}\leq2+\frac{1}{\log2}\leq4$. For the second part of the theorem, we rearrange (\ref{ene-decay-expresult-alpha}) as
    \begin{align*}
        2\sqrt{\frac{\eta_{k}}{k}}(\E[f(x_{k})]-f^*)
        \leq\E[\ene(k-1)]-\E[\ene(k)]+\frac{4\eta_{k}}{k}\sigma^2,
    \end{align*}
    and then sum over $k$ to get
    \begin{align}\label{ene-decay-expresult-alpha-cont}
        \begin{split}
            \sum_{i=1}^k\frac{1}{L}\sqrt{\frac{1}{i\log^2(i+2)}}(\E[f(x_i)]-f^*)
            \leq&\ene(0)-\E[\ene(k)]+\frac{\sigma^2}{L^2}\sum_{i=1}^k\frac{1}{i\log^2(i+2)}\\
            \leq&\ene(0)+\frac{\sigma^2}{L^2}\sum_{i=1}^\infty\frac{1}{i\log^2(i+2)}.
        \end{split}
    \end{align}
    The RHS of (\ref{ene-decay-expresult-alpha-cont}) is a finite number, and we denote it by $D$ for convenience. Suppose that there exists $\mu\geq0$ and $N\geq1$ such that $\E[f(x_i)]-f^*\geq\frac{\mu}{\sqrt{i}\log\log(i+2)}$ for any $i\geq N$, then we choose $k\geq N$ in (\ref{ene-decay-expresult-alpha-cont}), and the LHS is lower bounded by $\frac{\mu}{L}\sum_{i=N}^k\frac{1}{i\log(i+2)\log\log(i+2)}$. Substitute this inequality to (\ref{ene-decay-expresult-alpha-cont}) and take limit on both sides yield that
    \begin{align*}
        \frac{\mu}{L}\lim_{k\to\infty}\sum_{i=N}^k\frac{1}{i\log(i+2)\log\log(i+2)}\leq D,
    \end{align*}
    which leads to a contradiction since 
    \begin{equation*}
        \lim_{k\to\infty}\sum_{i=N}^k\frac{1}{i\log(i+2)\log\log(i+2)}=\infty.
    \end{equation*}
    Therefore, we finish the proof of (\ref{discrete-result-exp2}).
\end{proof}

\subsection{Almost Sure Convergence of SGDM}\label{sec:as}
By using a classical supermartingale convergence theorem, Lemma \ref{lem:explicit-decay-difference-alpha} also leads to the following almost sure convergence result.
\begin{theorem}\label{thm:as}
    Let Assumptions \ref{assump:convex-smooth} and \ref{assump:sto-gradient} hold. If $\{r_k\}$ satisfies $\lim_{k\to\infty}r_k=0$, then for $\{x_k\}$ generated by SGDM (\ref{eq:SGDM}) with stepsize $\eta_k=\frac{1}{L^2\log^2(k+2)}$ and initialization $x_0=x_1$, we have almost surely that
    \begin{equation*}
        \lim_{k\to\infty}r_k\cdot\frac{\sqrt{k}}{\log k}\cdot(f(x_k)-f^*)=0,
    \end{equation*}
    and
    \begin{equation*}
        \liminf_{k\to\infty}\bigg(\sqrt{k}\log\log(k+2)\bigg)\bigg(f(x_k)-f^*\bigg)=0.
    \end{equation*}
\end{theorem}

\begin{proof}
    The proof of Theorem \ref{thm:as} relies on the following supermartingale convergence result in \citet{robbins1971convergence}.
    \begin{lemma}\label{lemma:r-s}
        Let $(\Omega,\mathcal{F},\Pr)$ be a probability space and $\F_1\subset\F_2\subset\cdots$ a sequence of sub-$\sigma$-algebras of $\F$. For each $k=1,2,\cdots$, let $z_k$, $\xi_k$ and $\zeta_k$ be non-negative $\mathcal{F}_k$-measurable random variables such that $\sum_{i=1}^\infty\xi_i<\infty$ almost surely and
        \begin{align*}
            \E[z_{k+1}|\mathcal{F}_k]\leq z_k+\xi_k-\zeta_k,
        \end{align*}
        for any $k\geq1$. Then $\lim_{k\to\infty}z_k$ exists and is finite and $\sum_{i=0}^\infty\zeta_i<\infty$ almost surely.
    \end{lemma}
    In the proof of Theorem \ref{thm:algo-rate}, we have shown that with stepsize $\eta_k=\frac{1}{L^2\log^2(k+2)}$, the inequality
    \begin{align*}
        \E[\ene(k)-\ene(k-1)|\mathcal{F}_{k}]\leq\frac{4\eta_{k}}{k}\sigma^2-2\sqrt{\frac{\eta_{k}}{k}}(f(x_{k})-f^*)
    \end{align*}
    holds for any $k\geq1$. The definition yields that $\ene(k-1)$ and $2\sqrt{\frac{\eta_{k}}{k}}(f(x_{k})-f^*)$ are $\F_k$-measurable, and $\sum_{i=1}^\infty\frac{4\eta_{i}}{i}\sigma^2<\infty$. Therefore, we let $z_k=\ene(k-1)$, $\xi_k=\frac{4\eta_{k}}{k}\sigma^2$, and $\zeta_k=2\sqrt{\frac{\eta_{k}}{k}}(f(x_{k})-f^*)$, and thus applying Lemma \ref{lemma:r-s} gives
    \begin{enumerate}
        \item $\lim_{k\to\infty}\ene(k)=\ene_\infty$ exists with probability one;
        \item $\sum_{k=1}^\infty \sqrt{\frac{\eta_k}{k}}(f(x_{k})-f^*)$ exists with probability one.
    \end{enumerate}
    The definition of the Lyapunov function (\ref{discrete-energy:alpha}) yields that 
    \begin{align}\label{eq:as1}
        r_k\cdot\frac{\sqrt{k}}{\log k}\cdot(f(x_k)-f^*)\leq r_k\cdot\frac{\sqrt{k}}{\log k}\cdot\frac{\ene(k)}{4\sqrt{(k+1)\eta_k}}=r_k\cdot\frac{\sqrt{k}}{\log k}\cdot\frac{L\log(k+2)}{4\sqrt{(k+1)}}\cdot\ene(k).
    \end{align}
    Since $\lim_{k\to\infty}\ene(k)$ exists almost surely and $\lim_{k\to\infty}r_k=0$, (\ref{eq:as1}) yields that $\lim_{k\to\infty}r_k\cdot\frac{\sqrt{k}}{\log k}\cdot(f(x_k)-f^*)=0$ almost surely, and thus we finish the proof of the first part.

    For the second part of the theorem, we would show that 
    \begin{align}\label{eq:as2}
        \sum_{k=1}^\infty \sqrt{\frac{\eta_k}{k}}(f(x_{k})-f^*)<\infty
    \end{align}
    yields 
    \begin{align}\label{eq:as3}
        \liminf_{k\to\infty}\frac{f(x_k)-f^*}{\frac{1}{\sqrt{k}\log\log(k+2)}}=0.
    \end{align}
    Suppose that there exist $\mu\geq0$ and $N\geq1$ such that $f(x_i)-f^*\geq\frac{\mu}{\sqrt{i}\log\log(i+2)}$ for any $i\geq N$. Then we have
    \begin{align*}
        \sum_{k=1}^\infty \sqrt{\frac{\eta_k}{k}}(f(x_{k})-f^*)\geq\frac{\mu}{L}\sum_{k=N}^\infty\frac{1}{k\log(k+2)\log\log(k+2)}=\infty,
    \end{align*}
    which contradicts (\ref{eq:as2}). Since (\ref{eq:as2}) holds almost surely, this contradiction yields that (\ref{eq:as3}) also holds almost surely, and we complete the proof of Theorem \ref{thm:as}.
\end{proof}


\section{Anytime High Probability Convergence Result}\label{sec:hp}


In this section we present the anytime high probability convergence result for SGDM. Our analysis is based on the dynamical property illustrated by the discrete Lyapunov function $\ene(k)$ and Lemma \ref{lem:explicit-decay-difference-alpha}. We introduce a new approach to bound the increment $\ene(k)-\ene(k-1)$ by $\ene(k-1)$, the Lyapunov function of the previous step, which leads to a Gronwall-type argument. This technique is critical in removing the reliance on additional projection steps or assumptions of bounded gradients or stochastic gradients. Subsequently, by modifying the moment generating function of $\ene(k)$, we construct a supermartingale and establish the desired anytime high probability convergence. To begin with, we introduce the following assumption.


\begin{assumption}\label{assump:hp}
    For any $k\geq1$, the stochastic gradient $g(x_k,\xi_k)$ satisfies
    \begin{equation*}
        \E\left[\left.\exp\left(\frac{\|g(x_k,\xi_k)-\nabla f(x_k)\|^2}{\sigma^2}\right)\right|\mathcal{F}_k\right]\leq\exp(1).
    \end{equation*}
\end{assumption}
We now present our main convergence result.

\begin{theorem}\label{thm:real-anytime}
    Let Assumptions \ref{assump:convex-smooth}, \ref{assump:sto-gradient} and \ref{assump:hp} hold. For any $\beta$, the sequence $\{x_k\}$ generated by SGDM (\ref{eq:SGDM}) with stepsize $\eta_k=\frac{1}{16L^2\log^2(k+2)}$ and initialization $x_0=x_1$ satisfies
    \begin{align*}
        \Pr\left(f(x_k)-f^*\leq\frac{\left(C_1+C_2\log\frac{1}{\beta}\right)\log(k+2)}{\sqrt{k+1}},\;\text{for all $k\geq0$}\right)\geq1-2\beta,
    \end{align*}
    where $C_1=L\gamma_2\ene(0)+L\sigma^2(1+\sigma^2\gamma_1\gamma_2)\gamma_1$, $C_2=L\gamma_2+L\sigma^2(1+\sigma^2\gamma_1\gamma_2)\gamma_1$, $\gamma_1=\sum_{k=1}^\infty \frac{16\eta_k}{k}$, and $\gamma_2=\prod_{k=1}^\infty(1+\frac{16\eta_k}{k}\sigma^2)$.
\end{theorem}

Before proceeding with the proof, we explain the intuition of our approach. The Lyapunov function $\ene(k)$ and Lemma \ref{lem:explicit-decay-difference-alpha} provides the dynamical property of SGDM. More specifically, Lemma \ref{lem:explicit-decay-difference-alpha} yields that
    \begin{align}\label{eq:descent-property}
        \begin{split}
            \ene(k)-\ene(k-1)
        \leq&\frac{4\eta_{k}}{k}\lnorm g(x_{k},\xi_{k})\rnorm^2-\frac{2}{L}\sqrt{\frac{\eta_{k}}{k}}\|\nabla f(x_{k})\|^2
        +4\sqrt{\frac{\eta_{k}}{k}}\langle\nabla f(x_{k})-g(x_{k},\xi_{k}),\tau_{k}\rangle\\
        \leq&\frac{8\eta_{k}}{k}\lnorm \theta_k\rnorm^2+\frac{8\eta_k}{k}\|\nabla f(x_k)\|^2
        -\frac{2}{L}\sqrt{\frac{\eta_{k}}{k}}\|\nabla f(x_{k})\|^2+4\sqrt{\frac{\eta_{k}}{k}}\langle\theta_k,\tau_{k}\rangle\\
        \leq&a_k\lnorm \theta_k\rnorm^2+\sqrt{a_k}\langle\theta_k,\tau_{k}\rangle,
        \end{split}
    \end{align}
    where $\tau_k=k(x_k-x_{k-1})+(x_k-x^*)$, $\theta_k=\nabla f(x_k)-g(x_k,\xi_k)$, $a_k=\frac{16\eta_k}{k}$, and the second inequality follows from $\eta_{k}\leq\frac{k}{16L^2}$ for $k\geq1$.
Intuitively, the above inequality (\ref{eq:descent-property}) decouples the increment of $\ene(k)$
into two terms.
Classic concentration arguments bound the sum of the first part, $\sum_k a_k\|\theta_k\|^2$ with high probability. The second part, however, is more involuted. Noticing $\|\tau_k\|^2\leq \ene(k-1)$, we derive a Gronwall-type inequality and 
construct a supermartingale by modifying the moment generating function of $\ene(k)$. Consequently, we derive the desired anytime convergence result. To simplify the proof, we assume the stochastic terms $\{\xi_k\}$ are independent in the analysis below. This simplification can be removed by the martingale argument using conditional expectation.
\begin{proof}
    To prove Theorem \ref{thm:real-anytime}, we would present concentration estimates on $\Pr(\sup_{k\geq1}\ene(k)\geq \alpha)$. We denote $S(k)=\sum_{l=1}^ka_l\|\theta_l\|^2$ and $M(k)=\ene(k)-S(k)$. This gives
    \begin{align}\label{eq:diff-of-M}
        M(k)-M(k-1)=(\ene(k)-\ene(k-1))-(S(k)-S(k-1))
        \leq\sqrt{a_k}\langle\theta_k,\tau_{k}\rangle.
    \end{align}
    We write $\gamma_1=\sum_{k=1}^\infty a_k$, and $\gamma_2=\prod_{k=1}^\infty(1+a_k\sigma^2)$. It is direct to verify that these two constants are finite for $a_k=\frac{16\eta_k}{k}=\frac{1}{L^2k\log^2(k+2)}$.

    For the moment-generating function of $M(k)$, (\ref{eq:diff-of-M}) gives
    \begin{align*}
        \mathbb{E}[\exp(t M(k))|\mathcal{F}_{k-1}] &\leq \exp(t M(k-1)) \mathbb{E}[\exp (t \sqrt{a_k} \langle \theta_k,\tau_k\rangle) |\mathcal{F}_{k-1}].
    \end{align*}
    Definition of the Lyapunov function (\ref{discrete-energy:alpha}) yields that $\|\tau_k\|^2\leq\ene(k-1)$, and thus Lemma \ref{lemma:large-dev-3} gives
    \begin{align*}
        \mathbb{E}[\exp (t \sqrt{a_k} \langle \theta_k,\tau_k\rangle) |\mathcal{F}_{k-1}]\leq\exp(a_k\sigma^2t^2\ene(k-1)).
    \end{align*}
    Therefore, for $t\leq1$, we have
    \begin{align}\label{eq:mgf-decomposition}
        \begin{split}
            \mathbb{E}[\exp(t M(k))|\mathcal{F}_{k-1}]&\leq \exp(tM(k-1)+a_k \sigma^2 t^2 \mathcal{E}(k-1))\\
            &\leq \exp ((1+a_k \sigma^2)t M(k-1)+a_k \sigma^2 t S(k-1)),
        \end{split}
    \end{align}
    where the second inequality follows from $\ene(k-1)=M(k-1)+S(k-1)$ and $t\leq1$.

    We write
    \begin{align*}
        N^t(k)=\exp\left(\prod_{l=k+1}^{\infty}(1+a_l \sigma^2) t M(k)-\sum_{l=1}^{k}a_l\sigma^2\gamma_2 tS(l-1)\right).
    \end{align*}
    Specifically, this definition gives $N^t(0)=\exp\left(\prod_{l=0}^\infty(1+a_l\sigma^2)tM(0)\right)=\exp(\gamma_2t\ene(0))$. The conditional expectation of $N^t(k)$ for $t\leq\frac{1}{\gamma_2}$ satisfies
    \begin{align}
        &\E[N^t(k)|\mathcal{F}_{k-1}]\nonumber\\
        =&\E\left[\left.\exp\left(\prod_{l=k+1}^{\infty}(1+a_l \sigma^2) t M(k)-\sum_{l=1}^{k}a_l\sigma^2\gamma_2 tS(l-1)\right)\right|\mathcal{F}_{k-1}\right]\nonumber\\
        =&\E\left[\left.\exp\left(\prod_{l=k+1}^{\infty}(1+a_l \sigma^2) t M(k)\right)\right|\mathcal{F}_{k-1}\right]\cdot\exp\left(-\sum_{l=1}^{k}a_l\sigma^2\gamma_2 tS(l-1)\right)\nonumber\\
        \leq&\exp\left(\prod_{l=k}^{\infty}(1+a_l \sigma^2) t M(k-1)+a_k\sigma^2\prod_{l=k+1}^{\infty}(1+a_l \sigma^2)tS(k-1)-\sum_{l=1}^{k}a_l\sigma^2\gamma_2 tS(l-1)\right)\label{eq:Ntk}\\
        =&N^t(k-1)\cdot\exp\left(a_k\sigma^2\prod_{l=k+1}^{\infty}(1+a_l \sigma^2)tS(k-1)-a_k\sigma^2\gamma_2 tS(k-1)\right)\nonumber\\
        \leq& N^t(k-1),\nonumber
    \end{align}
    where (\ref{eq:Ntk}) follows from  (\ref{eq:mgf-decomposition}). Therefore, $\{N^t(k)\}_{k=0}^\infty$ is a supermartingale for $0<t\leq\frac{1}{\gamma_2}$. Then applying Ville's inequality \citep{ville1939etude} to $\{N^t(k)\}_{k=0}^\infty$ gives
    \begin{align}\label{eq:doob}
        \Pr\left(\sup_{k\geq0}N^t(k)\geq\exp(\alpha t)\right)\leq\exp(-\alpha t)\E[N^t(0)]=\exp(-\alpha t+\gamma_2t\ene(0)).
    \end{align}
    Since $\gamma_1S(k)\geq\sum_{l=1}^ka_lS(l-1)$, we have
    \begin{align*}
        \exp\left(\frac{M(k)}{\gamma_2}-\sigma^2\gamma_1S(k)\right)\leq\exp\left(\frac{\prod_{l=k+1}^\infty(1+a_l\sigma^2)M(k)}{\gamma_2}-\sum_{l=1}^ka_l\sigma^2S(l-1)\right)=N^{\frac{1}{\gamma_2}}(k),
    \end{align*}
    and thus
    \begin{align}\label{eq:anytime-final1}
        \begin{split}
            &\Pr\left(\sup_{k\geq0}\left\{M(k)-\sigma^2\gamma_1\gamma_2S(k)\right\}\geq\gamma_2\left(\ene(0)+\log\frac{1}{\beta}\right)\right)\\
            =&\Pr\left(\sup_{k\geq0}\left\{\exp\left(\frac{M(k)}{\gamma_2}-\sigma^2\gamma_1S(k)\right)\right\}\geq\exp\left(\ene(0)+\log\frac{1}{\beta}\right)\right)\\
            \leq&\Pr\left(\sup_{k\geq0}N^{\frac{1}{\gamma_2}}(k)\geq\exp\left(\ene(0)+\log\frac{1}{\beta}\right)\right)\leq\beta,
        \end{split}
    \end{align}
    where the last inequality follows from (\ref{eq:doob}) with $t=\frac{1}{\gamma_2}$ and $\alpha=\gamma_2\left(\ene(0)+\log\frac{1}{\beta}\right)$.

    On the other hand, Lemma \ref{lemma:large-dev-1} gives
    \begin{align*}
            &\Pr\left((1+\sigma^2\gamma_1\gamma_2)S(k)\geq\left(1+\log\frac{1}{\beta}\right)\sigma^2(1+\sigma^2\gamma_1\gamma_2)\gamma_1\right)\\
        \leq&\Pr\left((1+\sigma^2\gamma_1\gamma_2)S(k)\geq\left(1+\log\frac{1}{\beta}\right)\sigma^2(1+\sigma^2\gamma_1\gamma_2)\sum_{l=1}^ka_l\right)\leq\beta.
    \end{align*}
    Since $\{S(k)\}_{k=0}^\infty$ is monotonically increasing, we further have
    \begin{align}\label{eq:anytime-final2}
        &\Pr\left(\sup_{k\geq0}\left\{(1+\sigma^2\gamma_1\gamma_2)S(k)\right\}\geq\left(1+\log\frac{1}{\beta}\right)\sigma^2(1+\sigma^2\gamma_1\gamma_2)\gamma_1\right)\leq\beta.
    \end{align}

    Combining (\ref{eq:anytime-final1}) and (\ref{eq:anytime-final2}), we derive
    \begin{align*}
        &\Pr\left(\sup_{k\geq 0}\ene(k)\geq \gamma_2\left(\ene(0)+\log\frac{1}{\beta}\right)+\left(1+\log\frac{1}{\beta}\right)\sigma^2(1+\sigma^2\gamma_1\gamma_2)\gamma_1\right)\\
        \leq&  \Pr\left(\sup_{k\geq0}\left\{M(k)-\sigma^2\gamma_1\gamma_2S(k)\right\}\geq\gamma_2\left(\ene(0)+\log\frac{1}{\beta}\right)\right)\\
        &+\Pr\left(\sup_{k\geq0}\left\{(1+\sigma^2\gamma_1\gamma_2)S(k)\right\}\geq\left(1+\log\frac{1}{\beta}\right)\sigma^2(1+\sigma^2\gamma_1\gamma_2)\gamma_1\right)\\
        \leq&2\beta.
    \end{align*}
    Since $4\sqrt{(k+1)\eta_k}(f(x_k)-f^*)\leq\ene(k)$, we conclude that
    \begin{align*}
        \Pr\left(f(x_k)-f^*\leq\frac{\left(C_1+C_2\log\frac{1}{\beta}\right)\log(k+2)}{\sqrt{k+1}},\;\text{for all $k\geq0$}\right)\geq1-2\beta,
    \end{align*}
    where $C_1=L\gamma_2\ene(0)+L\sigma^2(1+\sigma^2\gamma_1\gamma_2)\gamma_1$, and $C_2=L\gamma_2+L\sigma^2(1+\sigma^2\gamma_1\gamma_2)\gamma_1$.
\end{proof}
\begin{remark}\label{rmk:better-log} 
    By slightly modifying the stepsize, the convergence rate in Theorem \ref{thm:real-anytime} can be improved to $\mathcal{O}\left(\frac{\log\frac{1}{\beta}\log^\frac{1+\varepsilon}{2}k}{\sqrt{k}}\right)$, where $\varepsilon$ is a positive number that can be arbitrarily small; see Appendix \ref{app:better-log} for more details. 
\end{remark} 

\begin{figure*}[ht]
    \centering
	\subfigure[Artificial Experiment]{
        \includegraphics[width=7cm]{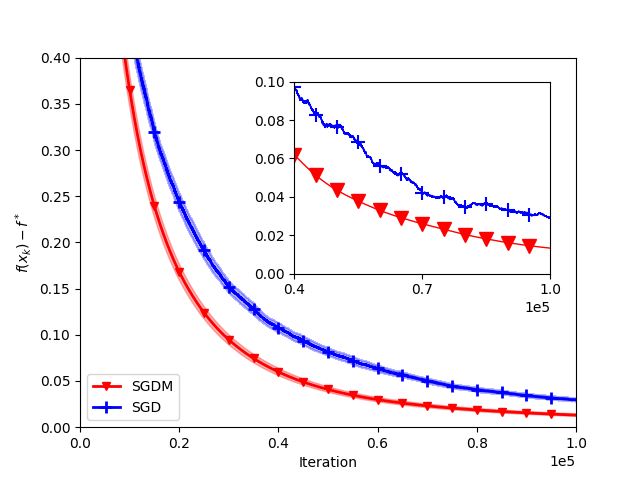}
		\label{fig:toy}
	}
        \subfigure[Logistic Regression]{
        \includegraphics[width=7cm]{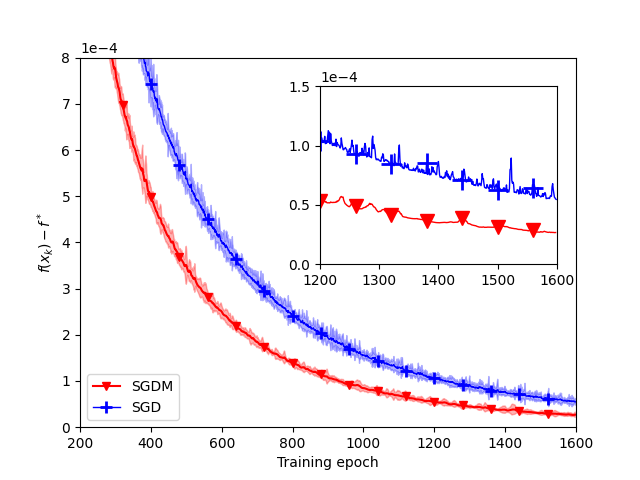}
		\label{fig:logreg}
	}\\
	\subfigure[MNIST]{
		\includegraphics[width=7cm]{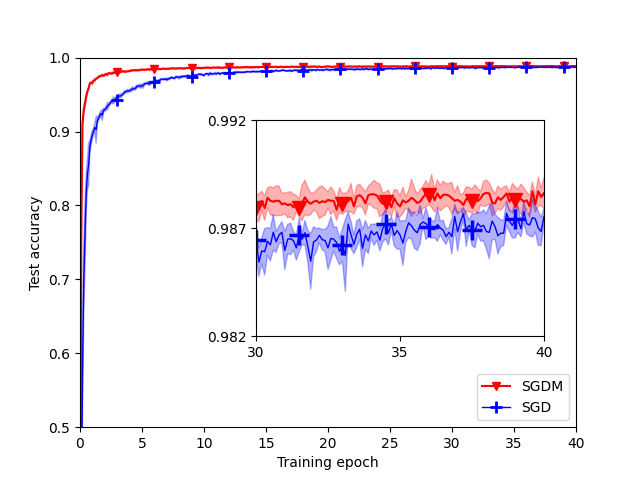}
		\label{fig:mnist}
	}
	\subfigure[CIFAR-10]{
		\includegraphics[width=7cm]{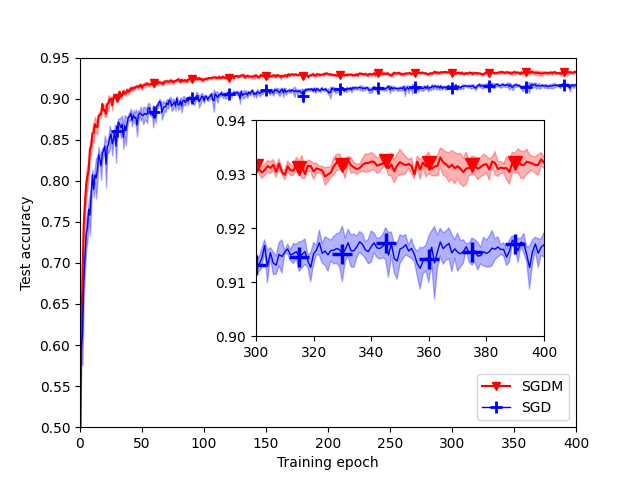}
		\label{fig:cifar}
	}
    \caption{Results of the artificial experiment (Figure \ref{fig:toy}), logistic regression (Figure \ref{fig:logreg}), image classification on MNIST (Figure \ref{fig:mnist}) and CIFAR-10 (Figure \ref{fig:cifar}). The solid lines in the main plot of all figures are obtained through averaging $10$ independent runs, and the variance regions are represented by shaded areas above and below the solid lines. The subplot in Figure \ref{fig:toy} and \ref{fig:logreg} presents an individual run, and the subplots in Figure \ref{fig:mnist} and \ref{fig:cifar} are zoomed in to provide better views of the last $1/4$ training processes.}
    \label{fig:adam}
\end{figure*} 

\section{Experiments} \label{sec:exp}
In this section, we test SGDM, in comparison with SGD, on an artificial example, a logistic regression problem, and two image classification tasks.

Firstly, we investigate the performance of SGDM on an artificial optimization problem and a logistic regression task. In the artificial experiment, the objective is a $10$-dimensional quadratic function, and a Gaussian noise is added to each gradient output. In the logistic regression task, the dataset is taken from LIBSVM library \citep{CC01a}, and the approximation of the minimal function value $f^*$ is obtained by a standard solver from \texttt{scipy} \citep{2020SciPy-NMeth}. For SGD, we use the classic stepsize $\eta_k=\frac{1}{\sqrt{k}}$. The results are presented in Figure \ref{fig:toy} and \ref{fig:logreg}. We observe that SGDM converges faster than SGD. Furthermore, the gradient noise in each iteration of SGDM is multiplied by a factor of order $\widetilde{\mathcal{O}}\left(\frac{1}{k^{\frac{3}{2}}}\right)$, while that of SGD is multiplied by a factor of order $\mathcal{O}\left(\frac{1}{\sqrt{k}}\right)$. This results in the error curve of SGDM being significantly smoother than SGD, as observed from the individual run in the subplot of Figure \ref{fig:toy} and \ref{fig:logreg}. Detailed settings about these two experiments and more plots of individual runs are presented in Appendix \ref{app:exp}.

Then we apply both SGDM and SGD to train neural networks for two image classification tasks, MNIST and CIFAR-10. The model used in the MNIST task is a two-layer CNN, and the model used in the CIFAR-10 task is Resnet-18 \citep{he2016deep}. The learning rate of SGD is tuned by the common hyperparameter selection practice \citep{bergstra2012random}. More details can be found in Appendix \ref{app:exp}. The results are presented in Figure \ref{fig:mnist} and \ref{fig:cifar}, which confirm the better performance of SGDM. Comparison of individual runs (presented in Appendix \ref{app:exp}) also shows that the learning curves of SGDM have smaller oscillations than SGD.


\section{Conclusion}

In this paper, we study the classic problem of stochastic first-order method for smooth convex programs. Despite numerous existing works, several fundamental questions remain unanswered. Among many open problems, one of the most important is: 

\begin{quote}
    \emph{Can we derive an anytime convergence rate without losing any factors in $k$, compared to the state-of-the-art convergence rate for a single-iterate guarantee? } 
\end{quote}

This paper gives an affirmative answer to this above question. This result is achieved by a careful analysis of the dynamics of stochastic gradient with momentum (SGDM). In particular, we establish a continuous-discrete time correspondence for SGDMs by a rigorous stochastic analysis.

This work paves the way for exploring possibilities in anytime analysis across various settings. For instance, it would be intriguing to investigate the potential establishment of a similar anytime convergence rate for non-convex programs.





\bibliography{last-iterate}

\appendix

\section{The continuous-time limit of the AC-SA algorithm \citep{lan2020first}}\label{app:acsa}
In this section, we give an informal derivation of the continuous-time limit of the well-known AC-SA algorithm, and show that it is also of the form (\ref{eq:SGDM-ODE-intro}).

For smooth convex objective function $f$ and Euclidean distance, the AC-SA algorithm \citep{lan2020first} takes the form
\begin{align*}
    &y_k=(1-\alpha_k)x_{k-1}+\alpha_kz_{k-1},\\
    &z_k=z_{k-1}-\gamma_kg(y_k,\xi_k),\\
    &x_k=(1-\alpha_k)x_{k-1}+\alpha_kz_k,
\end{align*}
where the parameters are given as $\alpha_k=\frac{2}{k + 1}$ and $\frac{1}{\gamma_k}=\frac{2L}{k}+\gamma\sqrt{k}$ \citep[Proposition 4.5]{lan2020first}. For sufficiently large $k$, the latter is the dominant term, so we approximately set $\gamma_k=\frac{1}{\gamma\sqrt{k}}$. Furthermore, we take $\eta = \frac{1}{\gamma^2}$ and introduce the continuous-discrete connection between $X(k\eta)$ and $x_k$, which plays the same role as the \emph{ansatz} $x_k\approx X(k\sqrt{s})$ in \citet{su2014differential}. As 
in the Theorem \ref{thm:sgdmtoode}, we let $g(y_k,\xi_k)=\nabla f(y_k)+\xi_k$ for each $k$, where $\xi_k\overset{iid}{\sim}\mathcal{N}(0,I_n)$. Substituting the above parameters, the AC-SA scheme becomes
\begin{align}
    &y_k=\frac{k-1}{k+1}x_{k-1}+\frac{2}{k+1}z_{k-1},\label{acsa-eq1}\\
    &z_k=z_{k-1}-\frac{\sqrt{\eta}}{\sqrt{k}}\left(\nabla f(y_k)+\xi_k\right),\label{acsa-eq2}\\
    &x_k=\frac{k-1}{k+1}x_{k-1}+\frac{2}{k+1}z_k.\label{acsa-eq3}
\end{align}
Rewriting (\ref{acsa-eq3}) as
\begin{align}\label{acsa-eq3-m}
    z_k=\frac{k+1}{2}x_k-\frac{k-1}{2}x_{k-1}
\end{align}
and substituting (\ref{acsa-eq1}) and (\ref{acsa-eq3-m}) to (\ref{acsa-eq2}) yields that
\begin{align*}
    &\frac{k+1}{2}x_k-\frac{k-1}{2}x_{k-1}\\
    =&\frac{k}{2}x_{k-1}-\frac{k-2}{2}x_{k-2}-\frac{\sqrt{\eta}}{\sqrt{k}}\left(\nabla f\left(\frac{k}{k+1}x_k+\frac{k-1}{k+1}x_{k-1}-\frac{k-2}{k+1}x_{k-2}\right)+\xi_k\right),
\end{align*}
and thus 
\begin{align*}
    x_k-x_{k-1} = \frac{k-2}{k+1}(x_{k-1}-x_{k-2})-\frac{2\sqrt{\eta}}{(k+1)\sqrt{k}}\left(\nabla f\left(x_k-\frac{1}{k+1}(x_k-x_{k-1})+\frac{k-2}{k+1}(x_{k-1}-x_{k-2})\right)+\xi_k\right).
\end{align*}
We let $v_k=\frac{x_{k}-x_{k-1}}{\eta}$, then the above equality gives
\begin{align}\label{acsa-eq-v}
    v_k = \frac{k-2}{k+1}v_{k-1}-\frac{2}{(k+1)\sqrt{k\eta}}\left(\nabla f\left(x_k-\frac{\eta}{k+1}v_k+\frac{(k-2)\eta}{k+1}v_{k-1}\right)+\xi_k\right).
\end{align}
Taylor expansion of $\nabla f$ yields that
\begin{align*}
    \nabla f\left(x_k-\frac{\eta}{k+1}v_k+\frac{(k-2)\eta}{k+1}v_{k-1}\right)=\nabla f(x_k)+\nabla^2f(\theta_k)\left(-\frac{1}{k+1}v_k+\frac{k-2}{k+1}v_{k-1}\right)\eta,
\end{align*}
where $\theta_k = \theta_k\left(x_k, x_k-\frac{\eta}{k+1}v_k+\frac{(k-2)\eta}{k+1}v_{k-1}\right)$. By reorganizing (\ref{acsa-eq-v}) and substituting the above Taylor expansion, the AC-SA scheme becomes
\begin{align}\label{eq:ACSAtoODE}
    \left\{\begin{aligned}
        x_k-x_{k-1}=&\eta v_{k},\\
        v_k-v_{k-1}=&-\frac{3}{(k-2)\eta}v_k\eta-\frac{2}{(k-2)\eta\sqrt{k\eta}}\nabla f(x_k)\eta-\frac{2}{(k-2)\eta\sqrt{k\eta}}(\sqrt{\eta}\xi_k)\sqrt{\eta}\\
        &-\frac{2}{(k-2)\eta\sqrt{k\eta}}\nabla^2f(\theta_k)\left(\frac{(k-3)\eta}{(k+1)\eta}v_k-\frac{(k-2)\eta}{(k+1)\eta}(v_k - v_{k-1})\right)\eta^2.
    \end{aligned}\right.
\end{align}

Let $W$ be a standard Brownian motion. As in Section \ref{sec:derivation-ode}, we let $\eta$ in (\ref{eq:ACSAtoODE}) go to zero and replace $x_k$, $v_k$, $\theta_k$, $x_k-x_{k-1}$, $v_k-v_{k-1}$, $k\eta$, $\eta$ and $\sqrt{\eta}\xi$ by $X$, $V$, $\Theta$, $\df X$, $\df V$, $t$, $\df t$ and $dW$, respectively. The discrete scheme (\ref{eq:ACSAtoODE}) becomes
\begin{align*}
    \left\{\begin{aligned}
        &\df X=V\df t,\\
        &\df V=-\frac{3}{t}V\df t-\frac{2}{{t}^\frac{3}{2}}\nabla f(X)\df t-\frac{2}{t^{\frac{3}{2}}}\sqrt{\df t}\df W-\frac{2}{t^\frac{3}{2}}\nabla^2 f(\Theta)\left(V-\df V\right)(\df t)^2.
    \end{aligned}\right.
\end{align*}
Then we omit the high order terms and obtain
\begin{align*}
    \left\{\begin{aligned}
        &\df X=V\df t,\\
        &\df V=-\frac{3}{t}V\df t-\frac{2}{{t}^\frac{3}{2}}\nabla f(X)\df t.
    \end{aligned}\right.
\end{align*}
Consequently, substituting the first line to the second line yields that
\begin{align*}
    \ddot{X}(t)+\frac{3}{t}\dot{X}(t)+\frac{2}{t^{3/2}}\nabla f(X(t))=0,
\end{align*}
which is also of the form (\ref{eq:SGDM-ODE-intro}).

\section{Additional related works}
\label{sec:additional}


Stochastic optimization has a long history, with its origins dating back to at least \citep{robbins1951stochastic,kiefer1952stochastic}. In recent years, the field has experienced significant growth, driven by the widespread use of mini-batch training in machine learning \citep{GoodBengCour16}, which induces stochasticity. Researchers have approached the problem of stochastic optimization from various perspectives. For instance, \citet{johnson2013accelerating}, \citet{allen2016variance}, \citet{allen2017katyusha}, \citet{allen2018katyusha}, and \citet{ge2019stabilized} investigated algorithms that only use full-batch training sporadically for finite-sum optimization problems. This class of algorithms are commonly known as Stochastic Variance Reduced Gradient (SVRG). 
\citet{li2019convergence}, \citet{ward2020adagrad}, and \citet{faw2022power} provided theoretical results for stochastic gradient methods with adaptive stepsizes.
\citet{zhang2020adaptive} demonstrated, both theoretically and empirically, that the heavy-tailed nature of gradient noise contributes to the advantage of adaptive gradient methods over SGD.
\citet{gadat2018stochastic, kidambi2018insufficiency, gitman2019understanding} studied the role of momentum in stochastic gradient methods. 
\citet{karimi2016linear} and \citet{madden2020high} proved the convergence of stochastic gradient methods under the PL condition. \citet{sebbouh2021almost} showed the almost sure convergence of SGD and the stochastic heavy-ball method.
The convergence of stochastic gradient methods in expectation is also well studied. \citet{shamir2013stochastic} studied the last-iterate convergence of SGD in expectation for non-smooth objective functions, and provided an optimal averaging scheme. \citet{yan2018unified} presented a unified convergence analysis for stochastic momentum methods. \citet{assran2020convergence} studied the convergence of Nesterov's accelerated gradient method in stochastic settings. \citet{liu2020improved} also provided an improved convergence analysis in expectation of a momentumized SGD, and proved the benefit of using the multistage strategy. 

\section{Auxiliary lemmas}\label{app:proof-thm-hp}

The proof of Theorem \ref{thm:real-anytime} relies on the following two auxiliary lemmas. These results have appeared in \citet{lan2012validation, devolder2011stochastic}, and we include the proofs for completeness.

\begin{lemma}\label{lemma:large-dev-3}
    Let $\theta_1,\cdots,\theta_k$ be a sequence of i.i.d. random variables, $\Gamma_l=\Gamma\left(\theta_{[l]}\right)$ and $\Delta_l=\Delta\left(\theta_{[l]}\right)$ be deterministic functions of $\theta_{[l]}=(\theta_1,\cdots,\theta_l)$, and $c_1,\cdots,c_k$ be a sequence of positive numbers such that:
    \begin{enumerate}
        \item $\E[\Gamma_l|\theta_{[l-1]}]=0$,\label{lem:c1}
        \item $|\Gamma_l|\leq c_l\Delta_l$,\label{lem:c2}
        \item $\E\left[\left.\exp\left(\frac{\Delta_l^2}{\sigma^2}\right)\right|\theta_{[l-1]}\right]\leq\exp(1)$,\label{lem:c3}
    \end{enumerate}
    hold for each $l\leq k$. Then, for any $\lambda\in\R$ and $l\geq1$, we have 
    \begin{align*}
        \E\lbm\left.\exp\lb\frac{\lambda\Gamma_l}{c_l\sigma}\rb\right|\theta_{[l-1]}\rbm\leq\exp\lb\frac{3\lambda^2}{4}\rb.
    \end{align*}
\end{lemma}
\begin{proof}
    Since $\exp(x)\leq x+\exp\lb\frac{9x^2}{16}\rb$ for any $x$, we have
    \begin{align*}
        \E\lbm\left.\exp\lb\frac{\lambda\Gamma_l}{c_l\sigma}\rb\right|\theta_{[l-1]}\rbm\leq& \E\lbm\left.\frac{\lambda\Gamma_l}{c_l\sigma}\right|\theta_{[l-1]}\rbm+\E\lbm\left.\exp\lb\frac{9\lambda^2\Gamma_l^2}{16c_l^2\sigma^2}\rb\right|\theta_{[l-1]}\rbm\\
        \leq& \E\lbm\left.\exp\lb\frac{9\lambda^2\Delta_l^2}{16\sigma^2}\rb\right|\theta_{[l-1]}\rbm,
    \end{align*}
    for any $\lambda$, where the second inequality uses conditions \ref{lem:c1} and \ref{lem:c2}. From the concavity of $f(x)=x^p$ for $0<p\leq1$, we further have
    \begin{align}\label{eq:lem-con1}
        \E\lbm\left.\exp\lb\frac{\lambda\Gamma_l}{c_l\sigma}\rb\right|\theta_{[l-1]}\rbm
        \leq
        \lb\E\lbm\left.\exp\lb\frac{\Delta_l^2}{\sigma^2}\rb\right|\theta_{[l-1]}\rbm\rb^\frac{9\lambda^2}{16}
        \leq
        \exp\lb\frac{9\lambda^2}{16}\rb,
    \end{align}
    for any $0<\lambda\leq\frac{4}{3}$, where the second inequality uses condition \ref{lem:c3}. On the other hand, since $\lambda x\leq\frac{3}{8}\lambda^2+\frac{2}{3}x^2$ for any $\lambda$ and $x$, we have
    \begin{align}
        \begin{split}\label{eq:lem-con2}
            \E\lbm\left.\exp\lb\frac{\lambda\Gamma_l}{c_l\sigma}\rb\right|\theta_{[l-1]}\rbm
            \leq&
            \exp\lb\frac{3\lambda^2}{8}\rb
            \E\lbm\left.\exp\lb\frac{2\Gamma_l^2}{3c_l^2\sigma^2}\rb\right|\theta_{[l-1]}\rbm\\
            \leq&
            \exp\lb\frac{3\lambda^2}{8}\rb
            \E\lbm\left.\exp\lb\frac{2\Delta_l^2}{3\sigma^2}\rb\right|\theta_{[l-1]}\rbm\\
            \leq&\exp\lb\frac{3\lambda^2}{8}+\frac{2}{3}\rb,
        \end{split}
    \end{align}
    for any $\lambda$, where the third inequality follows from the concavity of $f(x)=x^\frac{2}{3}$ and condition \ref{lem:c3}. Combining (\ref{eq:lem-con1}) with (\ref{eq:lem-con2}) yields, for any $\lambda$,
    \begin{align*}
        \E\lbm\left.\exp\lb\frac{\lambda\Gamma_l}{c_l\sigma}\rb\right|\theta_{[l-1]}\rbm\leq\exp\lb\frac{3\lambda^2}{4}\rb,
    \end{align*}
    and thus
    \begin{align*}
        \E\lbm\left.\exp\lb\lambda\Gamma_l\rb\right|\theta_{[l-1]}\rbm\leq\exp\lb\frac{3\lambda^2c_l^2\sigma^2}{4}\rb.
    \end{align*}
\end{proof}

\begin{lemma}\label{lemma:large-dev-1}
    Let $\theta_1,\cdots,\theta_k$ be a sequence of i.i.d. random variables, $\Phi_l=\Phi\left(\theta_{[l]}\right)$ be deterministic functions of $\theta_{[l]}=(\theta_1,\cdots,\theta_l)$, and $c_1,\cdots,c_k$ be a sequence of positive numbers. If the inequality
    \begin{align*}
        \E\left[\left.\exp\left(\frac{\Phi_l^2}{\sigma^2}\right)\right|\theta_{[l-1]}\right]\leq\exp(1)
    \end{align*}
    holds for each $l\leq k$, then for any $\Omega\geq0$, we have
    \begin{align*}
        \Pr\left(\sum_{l=1}^kc_l\Phi_l^2\geq(1+\Omega)\sum_{l=1}^kc_l\sigma^2\right)\leq\exp(-\Omega).
    \end{align*}
\end{lemma}

\begin{proof}
    Since $f(x)=e^x$ is convex, we have
    \begin{align*}
        \E\lbm\exp\lb\frac{\sum_{l=1}^kc_l\Phi_l^2}{\sum_{l=1}^kc_l\sigma^2}\rb\rbm=\E\lbm\exp\lb\frac{\sum_{l=1}^kc_l\sigma^2\frac{\Phi_l^2}{\sigma^2}}{\sum_{l=1}^kc_l\sigma^2}\rb\rbm
        \leq\frac{\sum_{l=1}^kc_l\sigma^2\E\lbm\exp\lb\frac{\Phi_l^2}{\sigma^2}\rb\rbm}{\sum_{l=1}^kc_l\sigma^2}.
    \end{align*}
    For each $l$, we have $\E\lbm\exp\lb\frac{\Phi_l^2}{\sigma^2}\rb\rbm=\E\lbm\E\lbm\left.\exp\lb\frac{\Phi_l^2}{\sigma^2}\rb\right|\theta_{[l-1]}\rbm\rbm\leq\exp(1)$, and thus
    \begin{align*}
        \E\lbm\exp\lb\frac{\sum_{l=1}^kc_l\Phi_l^2}{\sum_{l=1}^kc_l\sigma^2}\rb\rbm\leq\exp(1).
    \end{align*}
    Then Markov inequality gives, for all $\delta>0$,
    \begin{align*}
        \Pr\lb\exp\lb\frac{\sum_{l=1}^kc_l\Phi_l^2}{\sum_{l=1}^kc_l\sigma^2}\rb\geq\delta\rb\leq\frac{\exp(1)}{\delta}.
    \end{align*}
    Therefore, for any $\Omega\geq0$, by letting $\delta=\exp(1+\Omega)$ and using the monotonicity of $f(x)=e^x$, we conclude that
    \begin{equation*}
        \Pr\left(\sum_{l=1}^kc_l\Phi_l^2\geq(1+\Omega)\sum_{l=1}^kc_l\sigma^2\right)\leq\frac{\exp(1)}{\exp(1+\Omega)}=\exp(-\Omega).
    \end{equation*}
    Hence, we arrive at the desired concentration bound.
\end{proof}

\section{Proof of Remark \ref{rmk:better-log}}\label{app:better-log}

Here we show that the convergence rate in Theorem \ref{thm:real-anytime} can be improved to $\mathcal{O}\left(\frac{\log\frac{1}{\beta}\log^\frac{1+\varepsilon}{2}k}{\sqrt{k}}\right)$ by slightly modifying the stepsize.

\begin{theorem}
    Suppose that Assumptions \ref{assump:convex-smooth}, \ref{assump:sto-gradient} and \ref{assump:hp} hold. For any $\beta$, the sequence $\{x_k\}$ generated by SGDM (\ref{eq:SGDM}) with stepsize $\eta_k=\frac{1}{16L^2\log^{(1+\varepsilon)}(k+2)}$, where $0<\varepsilon\leq1$ and initialization $x_0=x_1$ satisfies
    \begin{align*}
        \Pr\left(\text{for any $k\geq0$},\;f(x_k)-f^*\leq\frac{\left(C_1+C_2\log\frac{1}{\beta}\right)\log^\frac{1+\varepsilon}{2}(k+2)}{\sqrt{k+1}}\right)\geq1-2\beta,
    \end{align*}
    where $C_1=L\gamma_2\ene(0)+L\sigma^2(1+\sigma^2\gamma_1\gamma_2)\gamma_1$, $C_2=L\gamma_2+L\sigma^2(1+\sigma^2\gamma_1\gamma_2)\gamma_1$, $\gamma_1=\sum_{k=1}^\infty \frac{16\eta_k}{k}$, and $\gamma_2=\prod_{k=1}^\infty(1+\frac{16\eta_k}{k}\sigma^2)$.
\end{theorem}

\begin{proof}
As in Section \ref{sec:hp}, we let $a_k=\frac{16\eta_k}{k}$, $\gamma_1=\sum_{k=1}^\infty a_k$, and $\gamma_2=\prod_{k=1}^\infty(1+a_k\sigma^2)$. It is straightforward to verify that these two constants are finite for $a_k=\frac{16\eta_k}{k}=\frac{1}{L^2k\log^{(1+\varepsilon)}(k+2)}$.

    The same argument as the proof of Theorem \ref{thm:real-anytime} yields that
    \begin{align*}
        \Pr\left(\sup_{k\geq 0}\ene(k)\geq \gamma_2\left(\ene(0)+\log\frac{1}{\beta}\right)+\left(1+\log\frac{1}{\beta}\right)\sigma^2(1+\sigma^2\gamma_1\gamma_2)\gamma_1\right)\leq2\beta.
    \end{align*}
    Since $4\sqrt{(k+1)\eta_k}(f(x_k)-f^*)\leq\ene(k)$, we conclude that
    \begin{align*}
        \Pr\left(\forall k\geq0,\;f(x_k)-f^*\leq\frac{\left(C_1+C_2\log\frac{1}{\beta}\right)\log^\frac{1+\varepsilon}{2}(k+2)}{\sqrt{k+1}}\right)\geq1-2\beta,
    \end{align*}
    where $C_1=L\gamma_2\ene(0)+L\sigma^2(1+\sigma^2\gamma_1\gamma_2)\gamma_1$, and $C_2=L\gamma_2+L\sigma^2(1+\sigma^2\gamma_1\gamma_2)\gamma_1$.
\end{proof}


\section{Additional Details and Results for the Experiments} \label{app:exp}

In our artificial experiment, the objective function  is $f(x)=\frac{1}{2}x^\tr Ax$, and $A\in\R^{10\times10}$ is randomly generated. When the algorithm calls the gradient oracle at $x\in\R^{10}$, it outputs a stochastic gradient $Ax+\xi$, where $\xi\sim\mathcal{N}(0, 100\cdot I_{10})$. 

For the logistic regression task, the objective function is
\begin{align*}
    f(\beta)=\frac{1}{N}\sum_{i=1}^N\left(y_i\log\frac{1}{1+\exp(-x_i^\tr\beta)}+(1-y_i)\log\left(1-\frac{1}{1+\exp(-x_i^\tr\beta)}\right)\right),
\end{align*}
and we use \texttt{australian} dataset from LIBSVM library \citep{CC01a}.

In the two image classification tasks, we set $\eta=0.2$ for SGDM. The stepsize of SGD is tuned over $[10^{-4},1]$ by grid search \citep{bergstra2012random}, and we report the results with the best choices, $0.005$ for the MNIST task and $0.01$ for the CIFAR-10 task. In the MNIST task, the model is trained for $40$ epochs, while in the CIFAR-10 task, the model is trained for $400$ epochs.

In Figures \ref{fig:individual-ae}, \ref{fig:individual-logreg}, \ref{fig:individual-mnist} and \ref{fig:individual-cifar} we present another three individual runs of SGDM and SGD for the four experiments. All these results confirm that SGDM has faster convergence and smaller oscillations than SGD.


\begin{figure*}[ht!]
    \centering
	\subfigure{
        \includegraphics[width=4.5cm]{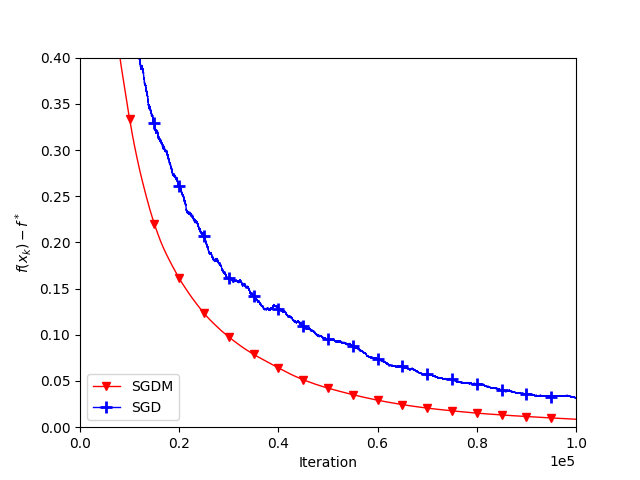}
	}
	\subfigure{
		\includegraphics[width=4.5cm]{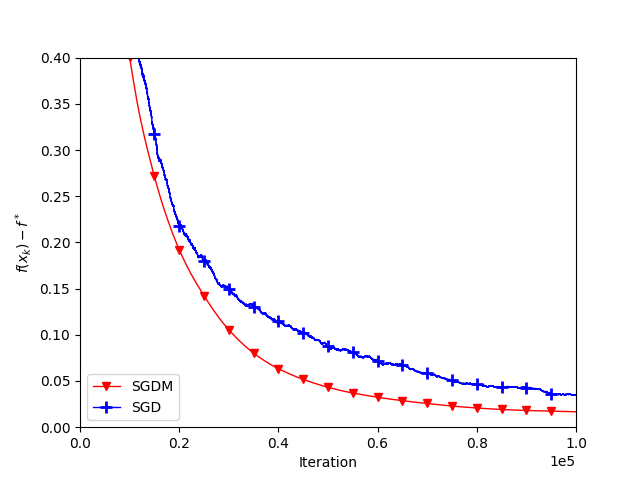}
	}
	\subfigure{
		\includegraphics[width=4.5cm]{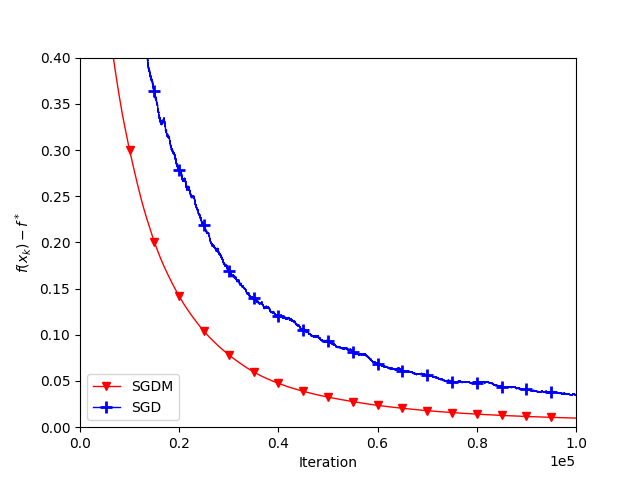}
	}
    \caption{Three individual runs of SGDM and SGD on the artificial experiment.}
    \label{fig:individual-ae}
\end{figure*} 
\begin{figure*}[ht]
    \centering
	\subfigure{
        \includegraphics[width=4.5cm]{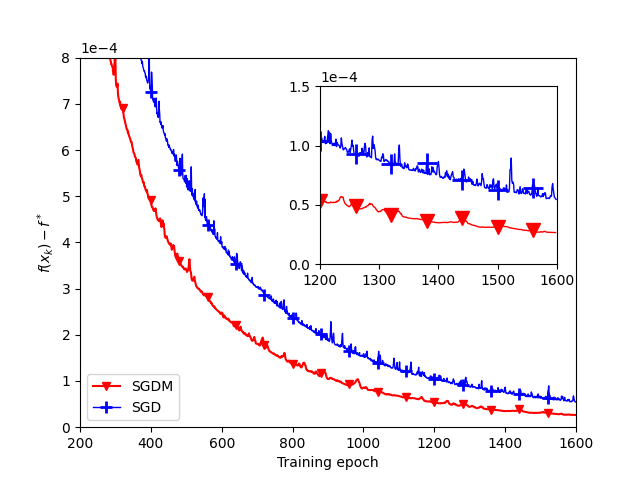}
	}
	\subfigure{
		\includegraphics[width=4.5cm]{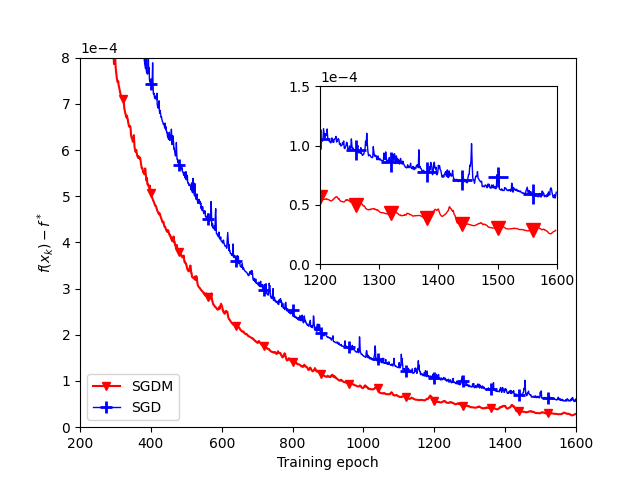}
	}
	\subfigure{
		\includegraphics[width=4.5cm]{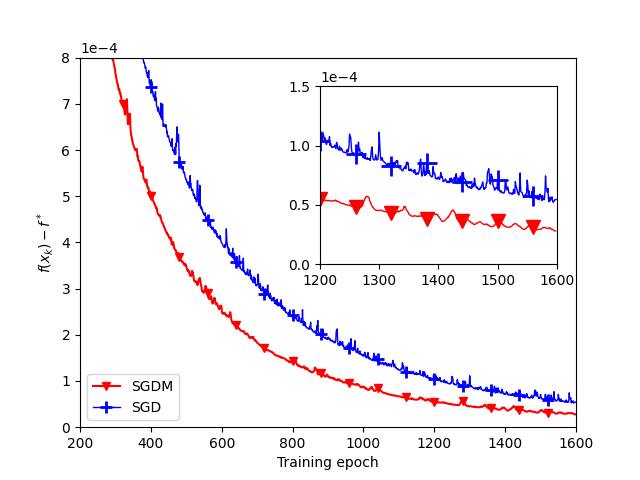}
	}
    \caption{Three individual runs of SGDM and SGD on the logistic regression task. The subplots are zoomed in to provide better views of the last $1/4$ training processes.}
    \label{fig:individual-logreg}
\end{figure*} 
\begin{figure*}[ht]
    \centering
	\subfigure{
        \includegraphics[width=4.5cm]{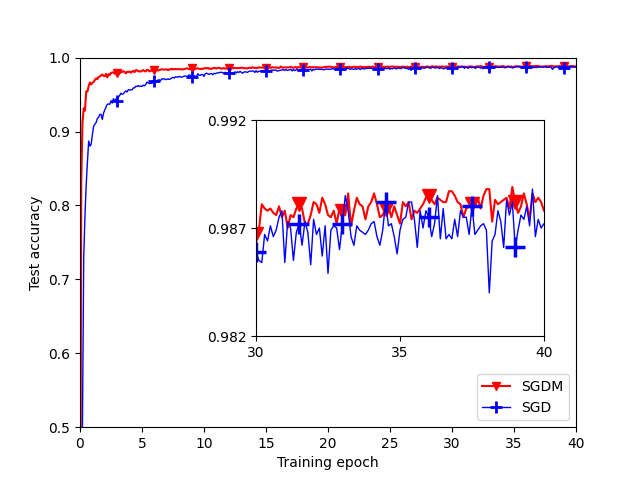}
	}
	\subfigure{
		\includegraphics[width=4.5cm]{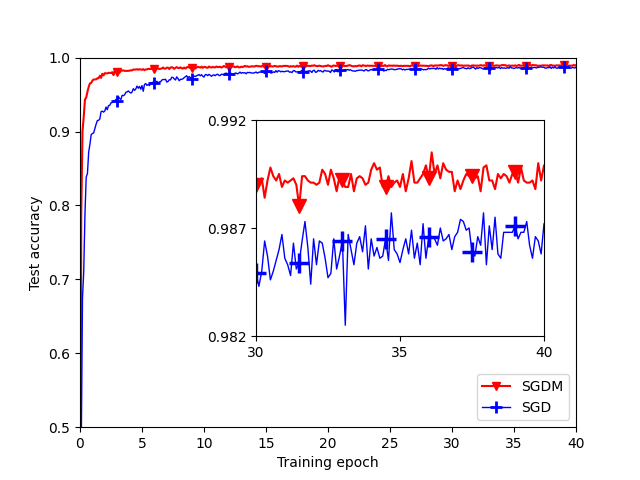}
	}
	\subfigure{
		\includegraphics[width=4.5cm]{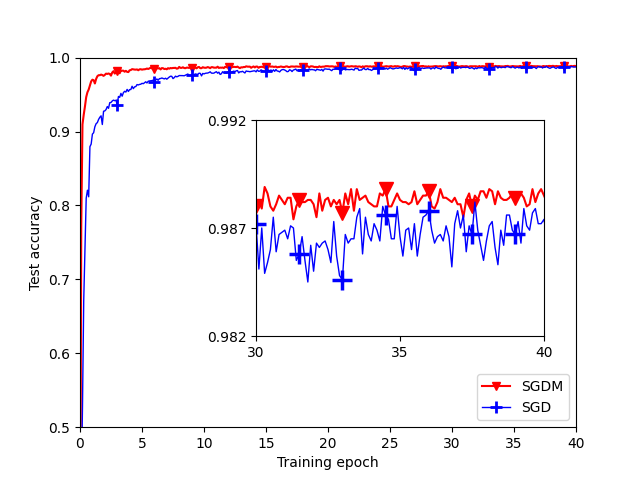}
	}
    \caption{Three individual runs of SGDM and SGD on the MNIST task. The subplots are zoomed in to provide better views of the last $1/4$ training processes.}
    \label{fig:individual-mnist}
\end{figure*} 
\begin{figure*}[ht]
    \centering
	\subfigure{
        \includegraphics[width=4.5cm]{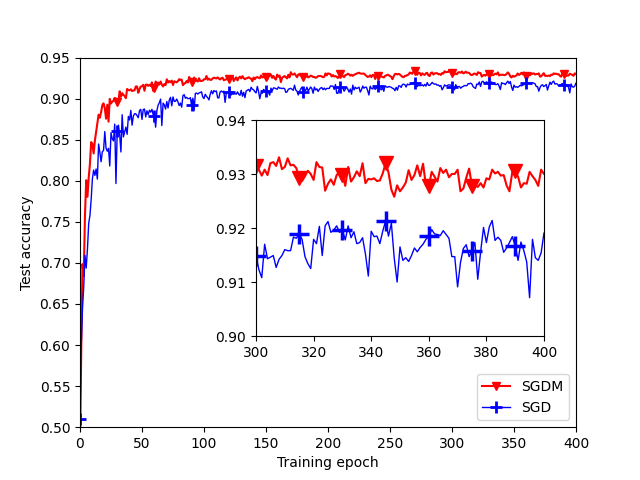}
	}
	\subfigure{
		\includegraphics[width=4.5cm]{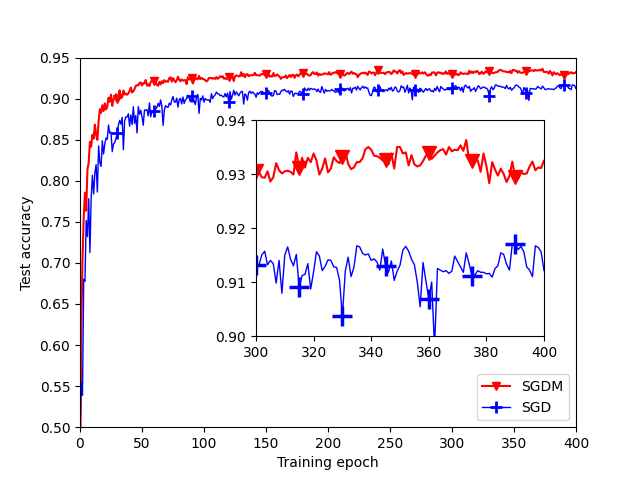}
	}
	\subfigure{
		\includegraphics[width=4.5cm]{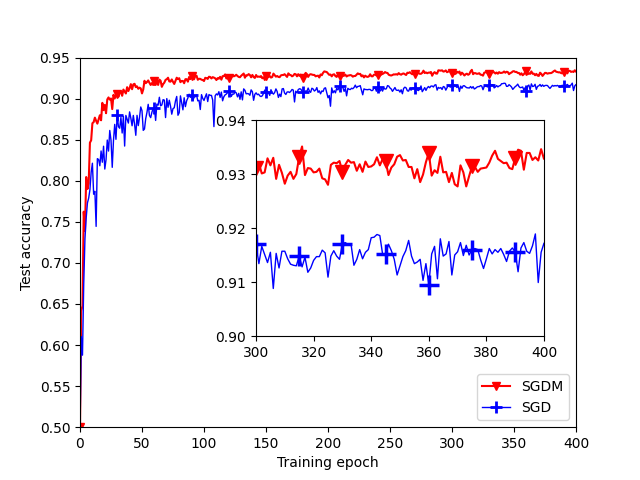}
	}
    \caption{Three individual runs of SGDM and SGD on the CIFAR-10 task. The subplots are zoomed in to provide better views of the last $1/4$ training processes.}
    \label{fig:individual-cifar}
\end{figure*}

\end{document}